\documentclass[a4paper]{elsarticle_n}
\usepackage[utf8]{inputenc}
\usepackage[margin=3cm]{geometry}
\usepackage{amsmath,amssymb,amsthm,bbm}
\usepackage[table]{xcolor}
\usepackage{xcolor,soul}
\usepackage{graphicx}
\usepackage{booktabs}
\usepackage{todonotes}
\usepackage{float}
\usepackage{latexsym}
\usepackage{bm}
\usepackage{tabu}
\usepackage{enumitem}
\usepackage{tcolorbox}
\usepackage{overpic}
\usepackage{accents}

\usepackage[normalem]{ulem}

\usepackage[draft,margin=false,inline=true,author=]{fixme}
\fxusetheme{colorsig}

\FXRegisterAuthor{gj}{agj}{GJ}
\FXRegisterAuthor{ai}{aai}{AI}
\FXRegisterAuthor{mtw}{amtw}{MTW}
\FXRegisterAuthor{ps}{aps}{PS}

\usepackage{pgfplots}
\pgfplotsset{compat=1.17}

\usepackage{hyperref}

\newtheorem{lemma}{Lemma}
\newtheorem{proposition}{Proposition}
\newtheorem{theorem}{Theorem}
\newtheorem{remark}{Remark}

\newcommand{\rJ}{\mathrm{J}}

    \journal{}

\title{Canards in a bottleneck}

\author[label1]{Annalisa Iuorio}\ead{annalisa.iuorio@univie.ac.at}
\author[label2]{Gaspard Jankowiak}\ead{gaspard@math.janko.fr}
\author[label3]{Peter Szmolyan}\ead{peter.szmolyan@asc.tuwien.ac.at}
\author[label4]{Marie-Therese Wolfram}\ead{m.wolfram@warwick.ac.uk}

\affiliation[label1]{organization={University of Vienna, Faculty of Mathematics},
            addressline={Oskar-Morgenstern-Platz 1}, 
            city={Vienna},
            postcode={1090},
            country={Austria}}
            
\affiliation[label2]{organization={Universität Konstanz, Fachbereich Mathematik und Statistik},
            addressline={Fach D 197}, 
            city={Konstanz},
            postcode={78457},
            country={Germany}}
            
\affiliation[label3]{organization={Technische Universität Wien, Institute for Analysis and Scientific Computing},
            addressline={Wiedner Hauptstr. 8-10},
            city={Vienna},
            postcode={1040},
            country={Austria}}
            
\affiliation[label4]{organization={University of Warwick, Mathematics Institute},
            city={Coventry},
            postcode={CV47AL},
            country={UK}}
\begin{document}
\listoffixmes{}

\begin{frontmatter}

% corridor with bottleneck

\begin{abstract}
In this paper we investigate the stationary profiles of a nonlinear Fokker-Planck equation with small diffusion and nonlinear
in- and outflow boundary conditions.
%characterisation of stationary solutions of an 1D area averaged PDE model for unidirectional pedestrian flows, which was investigated by Iuorio et al in\cite{Iuorio_2022}, to more general domains. 
We consider corridors with a bottleneck whose width has a global nondegenerate minimum in the interior. In the small diffusion limit the profiles are obtained constructively by using methods from geometric singular perturbation theory (GSPT). We identify three
main types of profiles corresponding to: (i) high density in the domain and a boundary layer at the entrance, (ii) low density in the domain and a boundary layer at the exit, and (iii) transitions from high density to low density inside the bottleneck with boundary layers at the entrance and exit.
%(\textcolor{red}{TODO: this is a bit crude, one should maybe mention boundary layers}).
Interestingly, solutions of the last type involve canard solutions generated at the narrowest point of the bottleneck.
We obtain a detailed bifurcation diagram of these solutions %which in addition also have boundary layers,
in terms of the in- and outflow rates. 
The analytic results based on GSPT are further corroborated by computational experiments investigating corridors with bottlenecks of variable width.

\end{abstract}

\end{frontmatter}

\section{Introduction}

In this paper we investigate the stationary profiles of a nonlinear Fokker-Planck equation with inflow and outflow boundary conditions, describing the unidirectional cross-sectional average flow of pedestrians in corridors with a single entrance and exit. Changes in the cross section lead to an in- or decrease of the possible flow inside the corridor; different in- and outflow conditions   to the formation of boundary layers at the entrance and exit. In \cite{Iuorio_2022} the authors derived the investigated 1D area averaged model from a nonlinear convection diffusion equation that was originally proposed by Burger and Pietschmann in \cite{burger_flow_2016}. They studied the formation of boundary layers in the case of strictly monotone cross sectional profiles using geometric singular perturbation theory (GSPT). In this paper we extend our analysis to corridors  with a 
unique point of minimal width which we denote as bottlenecks in the 
following.% turning point.\\

There has been an increased interest in the analysis of PDE models for pedestrian flows within the applied mathematics community in the last years. These models usually describe the dynamics of a single group of pedestrians having a common goal; for example unidirectional flows in corridor; or several groups with different objectives; as in bidirectional flows see \cite{BMP2011, BHRW2016}. The resulting PDEs or systems of PDEs are usually highly nonlinear and coupled. In addition to nonlinear boundary conditions, convection dominated terms as well as nonlinear interaction terms require the use of non-standard analytical and computational techniques to show existence of solutions, analyse their long time behavior and perform computational experiments. Stationary profiles of these PDE models provide useful insights into the complex dynamics and allow to predict seggregation dynamics (in the case of multi-species flows) or the formation of boundary layers or high density regions (in the case of low or high inflow and outflow rates or at bottlenecks), see for example \cite{BHRW2016, burger_flow_2016}. For a general overview on mathematical modeling, analysis and simulation we refer to \cite{CPT2014,MS2018}. \\

We reiterate that the investigated PDE model for area averaged flows comprises a nonlinear convection and linear diffusion term, as well as nonlinear in- and outflow at the entrance and exit. The interplay of small diffusion, the geometry of the domain as well as the in- and outflow rates lead to the formation of boundary layers, which we analyse using Geometric Singular Perturbation Theory (GSPT).\\

\noindent 
GSPT is a dynamical systems approach to singularly perturbed ordinary differential equations
started by the pioneering work of Fenichel \cite{Fenichel_1979}. The most common form of GSPT
considers slow-fast systems of the form
\begin{equation} \label{eq:slow}
 \begin{aligned}
  \dot{u} &= f(u,v), \\
  \varepsilon \dot{v} &= g(u,v),
 \end{aligned}
\end{equation}
where $u$ and $v$ are functions of $t$ and $0 < \varepsilon \ll 1$.
Often $t$ has the interpretation of time but it may represent equally well a spatial variable.
For $f =O(1)$ and $g =O(1)$ the variable $u$ varies on the slow time-scale $t$ and the variable $v$ on the fast time-scale
$\tau := \frac{t}{\varepsilon}$, which explains the name slow-fast system.
Written on the fast time-scale the equation has the form 
\begin{equation} \label{eq:fast}
 \begin{aligned}
  u' &= \varepsilon f(u,v), \\
  v' &= g(u,v).
 \end{aligned}
\end{equation}
Under suitable assumptions, solutions of System (\ref{eq:slow}) 
for small values of $\varepsilon$ can be constructed as perturbation of concatenations of solutions of the two limiting problems obtained by setting
$\varepsilon =0$ in systems (\ref{eq:slow}) and (\ref{eq:fast}), which are referred to as the
reduced problem and the layer problem, respectively.
In GSPT, these constructions are carried out in the framework of dynamical systems theory;
with the theory of invariant manifolds playing a particularly important role.
In the specific problem analysed in this paper, well established results and methods from GSPT are used and adapted for the analysis of a boundary value problem.

Therefore, we do not give a more detailed summary of GSPT, but refer to \cite{Jones_1995, Kuehn_2015} for more background on GSPT and its many applications. To name a few recent applications we mention the analysis of multi-scale structures in Micro-Electro-Mechanical Systems \cite{Iuorio_MEMS}
and in vegetation patterns \cite{Iuorio_2021}.
In the context of pedestrian dynamics, GSPT has been successfully applied to study closing channels in \cite{Iuorio_2022}.

The necessary concepts and results from GSPT are explained in Section 2 as needed in the context of the specific problem at hand.

\subsection{The mathematical model}

In the following we briefly discuss the underlying modeling assumption of the area averaged PDE under investigation. A more detailed derivation can be found in \cite{Iuorio_2022}.\\
We consider a undirectional flow of a large pedestrian crowd, whose density if given by $\rho = \rho(x,y,t)$, in a 2D domain with a single entrance (at $x=0$) and a single exit (at $x=L$). Furthermore we assume that the pedestrian density is constant across the cross section, that is for fixed $x$. This assumption is satisfied if
\begin{itemize}
\item the domain is symmetric with respect to the $x$-axis, and
\item the initial pedestrian distribution is symmetric with respect to the $x$-axis.
\end{itemize}
We assume that the dynamics are driven by convective transport and diffusion, in particular the total normalised pedestrian flow is given by
\begin{align}
    \label{eq:flux}
    \mathbf{j} = -\varepsilon \nabla \rho + \rho (1-\rho) \mathbf{u},
\end{align}
where $\mathbf{u}: \mathbb{R}^2 \rightarrow \mathbb{R}$ is a normalised vector field in the desired direction (in our case pointing in the general direction of the exit) and $\varepsilon > 0$ is the diffusion coefficient. We see that the average velocity corresponds to $1-\rho$, hence individuals move at maximum speed $1$ at density $\rho\equiv 0$ and vanishes if the density reaches its maximum value $\rho\equiv 1$. Note that the relation of the average density to the average velocity is commonly referred to as the fundamental diagram, and that similar relations have been investigated in traffic flow; consider for example the well known Lighthill-Whitham-Richard model \cite{lighthill1955kinematic, richards1956shock}.\\

In \cite{Iuorio_2022} the authors derived a 1D area averaged PDE model, which is based on the above assumptions and a suitable rescaling in space. It reads as
\begin{subequations}
\label{eq:1dareaaveraged}
\begin{align}
\label{eq:1edfpe}
    \partial_t \rho(x,t) = \partial_x \left(k(x)(-\varepsilon \partial_x \rho(x,t) + \rho(x, t)\, (1-\rho(x,t)) )\right) = 0
\end{align}
The %1D area averaged model
equation is supplemented with in- and outflow conditions 
\begin{align}
j(0,t) &= \alpha (1-\rho(0,t))\label{eq:in}\\
j(L,t) &= \beta \rho(0,t) \label{eq:out}
\end{align}
\end{subequations}
where $j(x, t) = -\varepsilon \partial_x \rho(x,t) + \rho(x,t) \, (1-\rho(x,t)) $
is the 1D equivalent of \eqref{eq:flux}.
In the derivation of \eqref{eq:1dareaaveraged}, the function $k$ is the product of the width with the cross-sectional average of first component of $\mathbf{u}$. For simplicity we refer to $k$ as the width of the bottleneck, which amounts to assuming that the cross-sectional average is 1. The parameters $\alpha > 0$ and $\beta > 0$ are the inflow and outflow rate, respectively. \\
The boundary condition \eqref{eq:in} describes the inflow at the entrance; the inflow is maximal if the entrance is empty ($\rho \equiv 0$), but decreases to zero when approaching the maximum density $\rho \equiv 1$. At the exit \eqref{eq:out} we do not assume that the outflow is limited by the maximum capacity. Hence, the outflow rate is proportional to the density of individuals at the exit.\\
In \cite{Iuorio_2022} the existence of a unique stationary solution of 
(\ref{eq:1dareaaveraged}) has been established in great generality by PDE methods, thus, it remains to understand its structure and dependence on parameters.

\subsection{Content and organisation of the paper}

In this paper we continue and extend the analysis of stationary profiles for system \eqref{eq:1dareaaveraged} in \cite{Iuorio_2022}, where a detailed analysis of 
stationary profiles and their dependence on the in- and outflow rates $\alpha$ and $\beta$ was given for corridors with monotonically decreasing (or increasing) functions $k$.
Recall, that smaller values of the function $k$ account for reduced mobility 
in narrower regions. It was shown in \cite{Iuorio_2022} that the nonlinear in- and outflow conditions lead to the formation of boundary layers at the entrance or exit.

Building on the approach in \cite{Iuorio_2022} we now consider the important case of corridors, whose width has a unique minimum. This setting corresponds to functions $k$ which have a unique global minimum at $x=x_0 \in (0,L)$. In the following we refer to domains of this type as corridors with bottlenecks (or sometimes only bottlenecks). Throughout this paper we will without loss of generality, assume that $L=1$. Therefore, we investigate the stationary states of system \eqref{eq:1dareaaveraged} described by
\begin{subequations}
\begin{equation}
    \partial_x \rJ = \partial_x(k(x) j(x)) = 0\,,
    \label{eq:reduced equation rho}
\end{equation}
where $j(x) = -\varepsilon \partial_x \rho(x) + \rho(x) \, (1-\rho(x)) $ coupled with the following boundary conditions
\begin{equation}
\begin{aligned}
    j &= \alpha \left(1 - \rho\right)
 & \text{ at } x &= 0\,,
 \\
    j &= \beta \rho
 & \text{ at } x &= 1\,.
\end{aligned}
\label{eq:boundary conditions reduced rho}
\end{equation}
\label{eq:reduced equation}
\end{subequations}
We characterise all profiles for different inflow and outflow rates $\alpha$ and $\beta$ in the singular limit $\varepsilon =0$. We identify 8 regions in parameter space corresponding
to profiles with different structures. Two of these regions correspond to high-density
profiles, two other regions correspond to low density profiles. These profiles
are quite similar to profiles considered in \cite{Iuorio_2022} and are only weakly affected
by the presence of the bottleneck. Due to the bottleneck a new  interesting class of  profiles exists, which corresponds
to solutions starting at high density and making a transition to low density in the region where the function $k$ attains its minimum. Since this type of solutions allows four possible configurations of boundary layers this leads to four types of transitional density profiles. We refer to these four types of profiles as transitional profiles.

Our GSPT analysis shows that these transitional profiles are caused by the existence of canard solutions
passing through a folded saddle \cite{Szmolyan_2001}. 
Canard solutions are solutions of singularly perturbed ODEs which
follow repelling slow manifolds for a considerable time. The essence of the canard phenomenon
is that these solutions lie exponentially close to the repelling slow manifold and
are therefore able to follow it for some time before they are ultimately repelled from it.
Clearly, special mechanisms are needed to bring solutions of interest exponentially close
to the repelling slow manifold.
The occurrence of canard solutions in boundary value problems
is conceptually less surprising than their occurrence in initial value problems, 
nevertheless we are not aware of similar works or results in the context of 
nonlinear boundary value problems.

The three profile types have a similar structure as the low density, high density and maximum current phases observed in Totally Asymmetric Simple Exclusion Process (TASEP). Note that the proposed model \eqref{eq:1dareaaveraged} was derived from a 2D TASEP, see \cite{burger_flow_2016}. In the TASEP, $\alpha$ and $\beta$ are the entry and exit rates, respectively, see for example \cite{derrida_exact_1992,wood_totally_2009}. \\
%\redd{ TODO: I do not really like the following sentence, it is also misplaced here! Our analysis also shows that for $\varepsilon > 0$, solutions get closer to the manifold of singular solutions as $\varepsilon \rightarrow 0$, except on a null set in the parameter space.}\\

%Complementing this analytical work, we illustrate and confirm ou results with computational %experiments for corridors with bottlenecks of variable width.\\

The rest of the paper is organised as follows. The GSPT analysis leading to the main result
on the structure of solutions is carried out in Section \ref{sec:gspt}.
In Section \ref{sec:numerics} the analytical results are illustrated and
confirmed by computational experiments for different channels. We conclude with an interpretation of the main features of the constructed solutions in the various regimes
in a manner which could be useful in further studies of pedestrian dynamics.   

\section{GSPT analysis}
\label{sec:gspt}

In this section, the stationary states associated to \eqref{eq:reduced equation} are investigated in a bottleneck scenario. The problem is rewritten as an equivalent boundary
value problem for an autonomous three-dimensional system of
first order differential equations in slow-fast form.
As explained in the introduction, we will identify  8 regions 
in the $(\alpha, \beta)$ parameter space, in which the stationary profiles have the same structure in the singular limit $\varepsilon =0$. We construct singular solutions of the boundary value problem as concatenations of solutions of the corresponding layer- and reduced problem. These singular solutions are then shown to persist for $\varepsilon$ small.
The profiles which exist in four of these regions involve 
a canard solution generated at a point corresponding to
the minimum of $k$.
%We reiterate that by extending our results to bottleneck scenarios leads to the emergence of two additional regions connected with canards, which were not observed in the previous work of \cite{Iuorio_2022}.\\

For the rest of this paper we make the following assumption which is crucial for our approach and results.\\[1mm]
{\bf Main Assumption:} The function $k \in C^2([0,1])$ is 
positive and has a unique global nondegenerate minimum at $x = x^\ast \in (0,1)$ satisfying
\begin{equation} \label{k-assumption}
%\frac{dk(x^\ast)}{dx} = 0, \qquad \frac{d^2k(x^\ast)}{dx^2} >0.
k'(x^\ast) = 0, \qquad k''(x^\ast) >0.
\end{equation}
\begin{remark}
Here we denote the derivative of the coefficient function $k$ as $k'$. Below we will also consider the function $g := k'/k$, its derivative will also be denoted as $g'$.
We would like to point out that starting with Equation \eqref{eq:sys_fast} the symbol $'$ will be mainly used to denote derivatives of the sought solution with respect to a rescaled fast variable. The above slight use of notations should not lead to any confusion.
\end{remark}

%In addition we assume
%$\frac{dk(x)}{dx} \leq 0$ for $x < x^\ast$ and  $\frac{dk(x)}{dx} \geq 0$ for $x > %x^\ast$. 
%This excludes other local minima or maxima of $k$ 
%in $(0,x^\ast)$ or $(x^\ast, 1)$ but allows that $k$ is 
%constant on some intervals not containing $x^\ast$, corresponding to regions of constant %width of the corridor.}

%\redd{
%\begin{remark}\label{k-remark}
%We expect that our results are valid for more general 
%functions $k$, i.e. the monotonicity assumptions on $k$ to the %left and right of
%$x^\ast$ can be dropped,  i.e. other lokal minima or maxima
%of $k$ in $(0,1)$ can be allowed. 
%All non-degenerate minima of $\frac{dk(x)}{dx}$ generate canard %points (folded saddles), all nondegenerate maxima of %$\frac{dk(x)}{dx}$ generate folded
%centers, which do not generate canard solutions.
%However, the solution structure of the boundary value problem %depends only on the global nondegenerate minimum $x^\ast$ of %$k$.
%Loosely speaking,  the two (singular) canard solutions passing
%through the primary folded saddle are the boundary
%of solutions of the reduced problem on the critical manifold %playing a role in the
%boundary value problem.
%\end{remark}}

By introducing the function
\[
 g :=\frac{k'}{k}
\]
%which hence satisfies the same sign properties as $\frac{dk(x)}{dx}$,
we can rewrite Equation \eqref{eq:reduced equation} as the system
\begin{equation} \label{eq:gform}
\begin{aligned}
 \frac{dj}{dx} &= -g(x)j, \\
 \varepsilon\frac{d\rho}{dx} &= \rho(1-\rho)-j.
\end{aligned}
\end{equation}
Analogously to \cite{Iuorio_2022}, this system can be transformed into an autonomous system by introducing the variable $\xi=x$ as a new dynamic variable and including the trivial equation $\frac{d \xi}{dx}=1$. From now on, we use the notation $\dot{~} = \frac{d}{dx}$. Thus, we obtain the following autonomous reformulation of \mbox{Equation \eqref{eq:gform}} 
\begin{equation} \label{eq:sys}
\begin{aligned}
 \dot{j} &= - g(\xi) j, \\
 \dot{\xi} &= 1, \\
 \varepsilon \dot{\rho} &= \rho(1-\rho)-j,
\end{aligned}
\end{equation}
where the above assumptions on $k$ identically apply with $\xi^\ast = x^\ast$,
%where $k(\xi) \in C^2([0,1])$ is a positive function which admits one global nondegenerate minimum\footnote{I.e. $\ddt{k}(\xi^\ast) > 0$.} at $\xi = \xi^\ast \in (0,1)$ and $\dt{k}(\xi) \leq 0$ for $\xi < \xi^\ast$, $\dt{k}(\xi) \geq 0$ for $\xi > \xi^\ast$, and $g(\xi)=\frac{\dot{k}(\xi)}{k(\xi)}$ (hence $g(\xi)=0$ if and only if $\xi=\xi^\ast$), 
with boundary conditions
\begin{equation} \label{eq:bc}
\begin{aligned}
    j &= \alpha \left(1 - \rho\right)
 & \text{ at } \xi &= 0\,,
 \\
    j &= \beta \rho
 & \text{ at } \xi &= 1\,.
\end{aligned}
\end{equation}
System \eqref{eq:sys} is a slow-fast system, where the dynamics of $\rho$ occur on the fast scale,
%(since its dynamics evolve proportionally to %$\frac{1}{\varepsilon} \gg 1$),
while the dynamics of $j$ and $\xi$ take place on the slow scale. %\cite{Fenichel_1979,JoneS_c995,Kuehn_2015}.
By transforming to the fast variable $\chi = \frac{x}{\varepsilon}$, and using the notation $'=\frac{d}{d \chi}$, we can rewrite System \eqref{eq:sys} as
%\gjnote{The fact that $\chi$ is defined here but not used in system (6) below feels strange. $x$ also appears in the definition of $\chi$ but essentially nowhere else.}
\begin{equation} \label{eq:sys_fast}
 \begin{aligned}
  j' &= -\varepsilon g(\xi) j, \\
  \xi' &= \varepsilon, \\
  \rho' &= \rho(1-\rho)-j.
 \end{aligned}
\end{equation}
As explained in the introduction, letting $\varepsilon \to 0$ in Equations \eqref{eq:sys} and \eqref{eq:sys_fast} leads to two limiting subproblems -- i.e.~the \emph{reduced} problem and the \emph{layer} problem, respectively -- which are simpler to analyse. The layer problem ($\varepsilon=0$ in \eqref{eq:sys_fast}) is given by
\begin{equation} \label{eq:laypb_k}
 \begin{aligned}
  j' &= 0, \\
  \xi' &= 0, \\
  \rho' &= \rho(1-\rho)-j, 
 \end{aligned}
\end{equation}
and describes the dynamics of the fast variable $\rho$ for fixed $j$ and $\xi$ values. The manifold of its equilibria is known as the \emph{critical manifold}
 \begin{equation} \label{eq:crit_man}
 \mathcal{C}_0 := \left\{ (j,\xi,\rho)~:~j=\rho(1-\rho) \right\},
 \end{equation}
which is a folded surface  in $(j,\xi,\rho)$ space.
The critical manifold $ \mathcal{C}_0  $ is the union of two submanifolds $\mathcal{C}_0^a$ ($\rho > \frac12$) and $\mathcal{C}_0^r$ ($\rho < \frac12$) -- which are attracting and repelling, respectively -- and a line of fold points
 \begin{equation} \label{eq:foldL}
  F := \left\{ (j,\xi,\rho)~:~j=\frac14,~\rho=\frac12 \right\},
 \end{equation}
as shown in Figure \ref{fig:C0_ff}. 
Fenichel Theory \cite{Fenichel_1979} implies that away from the fold line $F$ the 
submanifolds $\mathcal{C}_0^a$ and $\mathcal{C}_0^r$ 
perturb to (non-unique) attracting and repelling slow manifolds  
$\mathcal{C}_{\varepsilon}^a$ and $\mathcal{C}_{\varepsilon}^r$ 
for $\varepsilon$ small.

If the reduced flow reaches the fold line $F$ transversally at a point $p \in F$, the point $p$  is a jump point where a transition to fast motion close to solutions of the layer problem occurs, see \cite{Szmolyan_2004}. At exceptional points $p \in F$  where this transversality condition is violated
solutions of the reduced flow may cross through $p$ from 
$\mathcal{C}_0^a$ to $\mathcal{C}_0^r$, or vice versa. Such solutions are called (singular) canards, the corresponding $p \in F$ is a canard point. The least degenerate canard points
have been classified and analysed by the blow-up method 
as folded saddles and folded nodes in \cite{Szmolyan_2001}. There it is shown that these (singular) canards persist as canard solutions, i.e. solutions 
corresponding to intersections of the slow manifolds $\mathcal{C}_{\varepsilon}^a$ and $\mathcal{C}_{\varepsilon}^r$ 
near $p$ for $\varepsilon$ small. Thus, the existence of
canard solutions provides a mechanism that solutions lying in (or exponentially close to)  the attracting slow manifold $\mathcal{C}_{\varepsilon}^a$ can be continued in (or exponentially close to) the repelling slow manifold $\mathcal{C}_{\varepsilon}^r$. The less counter-intuitive situation
that solutions lying in the repelling slow manifold can be continued in (or close to) the attracting slow manifold is also 
possible. Canard solutions of this second type are often referred to as faux canards.

\begin{figure}[htb]
    \centering
  	\begin{overpic}[scale=.5]{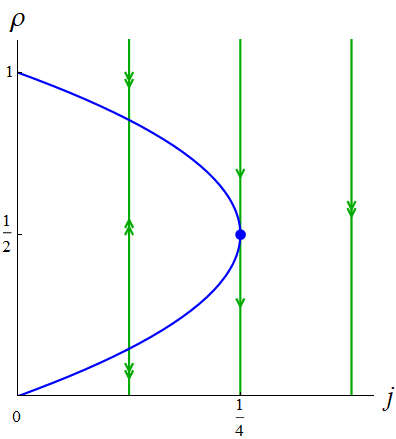}
  	\put(40,20){\footnotesize $\mathcal{C}_0^r$}
  	\put(40,70){\footnotesize $\mathcal{C}_0^a$}
  	\put(58,45){\footnotesize $F$}
    \end{overpic}
    \caption{Fast dynamics in $(j,\rho)$-space for a fixed value of $\xi$. The blue curve represents $\mathcal{C}_0$ 
    consisting of the  two branches $\mathcal{C}_0^a$ (attracting), $\mathcal{C}_0^r$ (repelling), and the fold line $F$. 
    The green lines indicate orbits of the layer problem \eqref{eq:laypb_k}, while the blue dot represents the line of fold points $F$.}
    \label{fig:C0_ff}
 \end{figure}
 
In the following we analyse the reduced flow on $F$. We will show that a canard point of folded saddle type occurs at the point 
\begin{equation} \label{eq:P}
 p^\ast=\left( \frac14, \xi^\ast, \frac12 \right)
\end{equation}
where $\xi^\ast$ is the location of the global minimum of the function $k$.

The reduced problem is very simple
\begin{subequations} \label{eq:k_redpb}
 \begin{align}
  \dot{j} &= -g(\xi) j, \label{eq:kj_const} \\
  \dot{\xi} &= 1.
 \end{align}
\end{subequations}
The phase space for the reduced problem is $[0, 1/4] \times [0,1]$, where $j= 1/4$ corresponds to the fold line.
It follows from Equation \eqref{eq:reduced equation rho} that
$k(\xi)\, j$ is a conserved quantity, hence the level lines of this function give the orbits of the reduced problem (\ref{eq:k_redpb}).
%, which can be lifted to the critical manifold $\mathcal{C}_0$ via its defining equation $\rho(1 - \rho )= j$.

However, as always for folded critical manifolds, the classification of the reduced flow -- in particular at the fold line and at canard points -- is more conveniently carried out in the variables $(\xi, \rho)$  by using the constraint $j = \rho(1 - \rho)$ which defines 
$\mathcal{C}_0$.
Differentiating the constraint with respect to $x$ gives 
$\dot{j} = (1-2 \rho) \dot{\rho}$, which allows to rewrite the reduced problem as 
\begin{equation} \label{eq:k_redpb_xrho}
\begin{aligned}
    \dot{\xi} &= 1, \\
    (1-2 \rho) \, \dot{\rho} &= -g(\xi) \, \rho (1-\rho).
    %\dt \xi &= 1, \\
    %\dt \rho &= -g(\xi)\frac{\rho (1-\rho)}{1-2 \rho}.
\end{aligned}
\end{equation}
with $\xi \in [0,1]$ and $\rho \in [0,1]$.
System \eqref{eq:k_redpb_xrho} is singular at the fold line $F$, i.e.~for $\rho=1/2$. This system can be desingularised 
by multiplying the right hand-side by $1-2\rho$
and dividing out this factor in the $\rho$ equation.
This gives the desingularised reduced system 
\begin{equation} \label{eq:red-desing}
\begin{aligned}
    \dot{\xi} &= 1 - 2 \rho, \\
    \dot{\rho} &= -g(\xi) \, \rho (1-\rho).
    %\dt \xi &= 1, \\
    %\dt \rho &= -g(\xi)\frac{\rho (1-\rho)}{1-2 \rho}.
\end{aligned}
\end{equation}
This multiplication of the right hand side by $(1 - 2 \rho)$ corresponds to a position dependent rescaling of the independent variable $x$,
which does not change orbits of the system away from the fold line.
However, for $\rho > 1/2$ the flow direction is reversed, which needs 
to be taken into account.

\medskip
We now collect the properties of the reduced problem, which are needed in the analysis of the boundary value problem \eqref{eq:sys}-\eqref{eq:bc}. 
These properties depend on properties of the function $g = k'/k$. Our main Assumption implies that the global nondegenerate minimum of $k$ at $\xi^\ast$ is a simple zero
of $g$ corresponding to a saddle point $(\xi^\ast, 1/2)$ of the
desingularised reduced problem.  Other zeros of $g$ lead to 
additional equilibria, which are discussed only briefly, since
we show later that these play no role in the analysis of the
boundary value problem.

%\bigskip
%\blu{SUGGESTION:\\
%Notation: Maybe we should call the fold line $F$? Calling it $S$  might be lead to confusion with the canards, see below:\\
% faux canard  $S_f$: reaches $\xi=0$ at $\rho = \rho_c^0 \in (0, 1/2)$,  reaches $\xi=1$ at $\rho =\rho_c^1 \in (0,1/2)$ \\
% canard orbit $S_c$:     reaches $\xi=0$ at $\rho = 1 -  \rho_c^0 \in (1/2,1)$, reaches $\xi=1$ at $\rho = 1- \rho_c^1 \in (1/2,1)$\\[2mm]
% TODO: change the notation used in Fig 5:  what is called  $S_c$ there is now called $S_f$, the $\tilde{S}_c$  is now called $S_c$.
% This means that we do not use the $\tilde{ }$ Notation for the canards, but that is good because the new $S_c$ is the important canard.
% I think that that discussion/exposition of the special solutions on page 10 has to be adapted, because the canards are now
% aleady taken care of.\\
% Fig 2 should focus on the canards and the phaseportrait\\
% Fig 5 on the  orbits $S_\alpha, S_\beta$ etc which are important for the boundary value problem.
%  }

\begin{lemma}\label{lem:reducedflow}
The reduced problem  \eqref{eq:k_redpb_xrho} has the following properties:
%(TODO:  or use a modification of Figure \ref{fig:sf_sp} to replace \ref{fig:bottlenecks}
\begin{enumerate}
\item The phase portrait is symmetric with respect to the line $\rho = 1/2$, which corresponds to the fold line $F$.
\item The variable $\xi$ is increasing along all orbits, i.e. the flow is from left to right.
\item The lines $\rho =0$ and $\rho=1$ are invariant.
\item  In regions with $g(\xi) >0$ the variable $\rho$ is decreasing along orbits for $1/2 < \rho< 1$ and is increasing for $0 < \rho < 1/2$.
In regions with $g(\xi)  < 0$ this monotonicity is reversed. The variable $\rho$ is constant in regions with $g(\xi) = 0$, corresponding to regions where
the width of the corridor is constant.
\item  The line $\rho =1/2$ is  a line of singularities. Points $(\xi, 1/2)$ with $g(\xi) > 0$ are reached in finite time by the forward flow and the derivative 
$\dot{\rho}$ blows up there. Similarly, points $(\xi, 1/2)$ with $g(\xi) < 0$ are reached in finite time by the backward flow.
\item
The point $ p^\ast =\left( \xi^\ast, \frac12 \right)$ is a canard point of folded saddle type.
\item There exist two (symmetric with respect to the line $\rho = 1/2$)  singular canard solutions with orbits $S_c$ and $\tilde{S}_c$ passing smoothly through the singularity  located at $p^\ast$.
The canard $S_c$ crosses from the attracting part of the critical manifold to the repelling one,
the (faux) canard $\tilde{S}_c$ crosses from the repelling part of the critical manifold to the attracting one.
\item The (faux) canard orbit $\tilde{S}_c$ starts at $\xi=0$, $\rho = \rho_c^0 \in (0, 1/2)$ and reaches $\xi=1$ at $\rho =\rho_c^1 \in (0,1/2)$.
The  canard orbit $S_c$  starts at  $\xi=0$, $\rho = 1 -  \rho_c^0 \in (1/2,1)$ and reaches $\xi=1$ at $\rho = 1- \rho_c^1 \in (1/2,1)$.  
\item Solutions starting at $\xi =0$ with $\rho \in [0,\rho_c^0)$ reach $\xi =1$ with  $\rho \in [0,1-\rho_c^1)$.
 Solutions starting at $\xi =0$ with $\rho \in (1-\rho_c^0,1]$ reach $\xi =1$ with  $\rho \in (\rho_c^1, 1])$.
\item Solutions starting at $\xi =0$ with $\rho \in (\rho_c^0 , 1 -   \rho_c^0 )$ do not cross the line $\xi = \xi^\ast$, in 
particular they do not reach the line $\xi =1$.  Solutions reaching  $\xi =1$ with $\rho \in (1- \rho_c^1 ,  \rho_c^0 )$ do not cross 
the line $\xi = \xi^\ast$ in backwards time, in particular they do not reach the line $\xi =0$.  
\item 
An isolated zero of $g$ at say $\xi_0 \neq \xi^\ast$
corresponds to another folded singularity at $(\xi_0,1/2)$,
which is a folded saddle for $g'(\xi_0) >0$ and a folded center for $g'(\xi_0) <0$. A more degenerate zero of $g$ corresponds
to a more degenerate folded singularity.
If $g$ is zero on an interval $[\xi_1, \xi_2]$, the density $\rho$ is constant there. In this situation $[\xi_1, \xi_2]\times \{1/2\} $ is a line of equilibria, the endpoints
of this line are again degenerate folded singularities. 
\end{enumerate}
\end{lemma}

The properties of the reduced problem described in the Lemma
are illustrated in Figure \ref{fig:bottlenecks}   
for a function $k$ which satisfies $k' < 0$ in $(1,\xi^\ast)$ and $k' > 0$ in $(\xi^\ast,1)$.

\begin{remark} ${ }$\\[-5mm]
\begin{enumerate}
    \item[(a)] The notation $S_c$ and $\tilde{S}_c$ for the canard orbits is chosen to be consistent with the notation we introduce
below for other orbits of the reduced problem in the construction of singular solutions of the boundary value problem.
\item[(b)] The property 11. associated with additional zeros of $g$ (which may occur under our rather general main 
Assumption on the function $k$) are mainly included for completeness. 
In Remark~\ref{rem:regN} below, we show that they play no role in the construction of solutions of the boundary value problem, due to property 10. of the Lemma.
\end{enumerate}
\end{remark}

\begin{proof}
Properties 1.-5. follow directly from the equations. The point $p^\ast$ is an  equilibrium for the desingularised system \eqref{eq:red-desing}.
The matrix associated with the linearisation of \eqref{eq:red-desing} at $p^\ast$ is
\begin{equation}\label{lin-label}
A := \left ( \begin{array}{cc}
0 & -2 \\[1mm]
-\frac{g'(\xi^\ast)}{4} & 0
\end{array}
\right ).
\end{equation}
Since $g'(\xi^{\ast})  = k''(\xi^{\ast})  /  k(\xi^{\ast})$, the assumption $k''(\xi^{\ast})>0$ translates into  $g'(\xi^{\ast})>0$. 
This gives
$\det A = -\frac{g(\xi^{\ast})}{2}  < 0$, hence $p^\ast$ is a saddle point for \eqref{eq:red-desing} with associated smooth stable and unstable
manifolds. 
For the reduced problem  \eqref{eq:k_redpb_xrho} -- with the flow direction reversed for $\rho > 1/2$ --
the point $p^\ast$ is a folded saddle  \cite{Szmolyan_2001}.
Due to a cancellation of a simple zero on both sides of the $\rho$-equation in \eqref{eq:k_redpb_xrho}, 
the stable manifold of the saddle is now the (faux) canard $\tilde{S}_c$, corresponding to a smooth solution passing through through
the point $p^\ast$. Similarly, the  unstable manifold of the saddle becomes the canard $S_c$. This proves properties 6. and 7.\\
The conserved quantity $k(\xi) j$ of equation \eqref{eq:kj_const} translates into the conserved quantity
\begin{equation}\label{Eq:H}
 H(\xi, \rho) = k(\xi) \rho (1 - \rho) >0 
 \end{equation}
of the desingularised system \eqref{eq:red-desing}, i.e. the level lines of $H$ give the phase portrait.
The canard orbits $S_c$ and $\tilde{S}_c$ are the level lines $H(\xi, \rho) = \frac{k(\xi^{\ast})}{4} $. Since $k$ has its global minimum
at $\xi^{\ast}$, the canard orbits cannot  intersect the (fold) line $\rho =1/2$. Since in addition, the canard orbits cannot 
intersect the lines $\rho =0$, $\rho=1$ where $H=0$, the canard orbits extend to $\xi =0$ and $\xi =1$. Thus assertion 8.
follows, with $\rho_c^0 \in (0,1/2)$ and $\rho_c^1 \in (1/2,1)$  defined as the solutions of the equations
\[ H(0, \rho_c^0) = \frac{k(\xi^{\ast})}{4} , \qquad  H(1, \rho_c^1) = \frac{k(\xi^{\ast})}{4} . \]
The solutions described in Assertion 9. lie on level lines with  $H(\xi,\rho) >  H(\xi^{\ast}, 1/2)$,
the solutions described in Assertion 10. lie on level lines with  $H(\xi,\rho) < H(\xi^{\ast}, 1/2)$. Together with 8.~this
implies 9. and 10.
\end{proof}

The canard $S_c$ and the (faux) canard $\tilde{S}_c$ on $\mathcal{C}_0$ can be described as graphs  by means of the following functions
\begin{subequations} \label{eq:canorb}
\begin{align}
 \rho_c^+(\xi) &:= \frac12 \left(1+\sqrt{1-\frac{k\left(\xi^\ast\right)}{k(\xi)}} \right),\\
 \rho_c^-(\xi) &:=\frac12 \left( 1-\sqrt{1-\frac{k(\xi^\ast)}{k(\xi)}} \right),
\end{align}
\end{subequations}
as follows
\begin{equation} \label{eq:Scflow}
\begin{aligned}
    S_c &:= \left\{ (\xi, \rho) \, : \, 0 \leq \xi \leq \xi^\ast, \, \rho = \rho_c^+(\xi) \right\} \cup \left\{ (\xi, \rho) \, : \, \xi^\ast \leq \xi \leq 1, \, \rho = \rho_c^-(\xi) \right\}, \\
    \tilde{S}_c &:= \left\{ (\xi, \rho) \, : \, 0 \leq \xi \leq \xi^\ast, \, \rho = \rho_c^-(\xi) \right\} \cup \left\{ (\xi, \rho) \, : \, \xi^\ast \leq \xi \leq 1, \, \rho = \rho_c^+(\xi) \right\}.
\end{aligned}
\end{equation}

The values $\rho_c^0$ and $\rho_c^1$ introduced in Lemma \ref{lem:reducedflow} (corresponding to the $\rho$-values of the (faux) canard) at $\xi =0$ and $\xi=1$, respectively, 
are then given by
\begin{equation} \label{eq:sp_rhof}
 \rho_c^0 = \rho_c^-(0), \qquad
 \rho_c^1 = \rho_c^+(1).
\end{equation}

The points of the canard $S_c$ corresponding to $\xi=0$ and $\xi=1$ in $(j, \xi, \rho)$-space  which play an important role in the following analysis are 
\begin{equation}
\begin{aligned}
 p_c^0 &:= \left( \rho_c^0(1-\rho_c^0), 0, 1-\rho_c^0 \right),\\
 p_c^1 &:= \left( \rho_c^1(1-\rho_c^1), 1, 1-\rho_c^1 \right).
\end{aligned}
\end{equation}

\begin{remark} \label{rem:regN}
The function $g$ may have zeros $\xi \neq \xi^\ast$.
All these points $(\xi, 1/2)$  are equilibria of the desingularised system \eqref{eq:red-desing} but these 
equilibria and possible canard solutions associated with them are confined to the open region $\mathcal{N}$ bounded by $\tilde{S}_c$  from
below and by $S_c$   from above for $\xi < \xi^{\ast}$, and by  $S_c$ from
below and by $\tilde{S}_c$ from above for $\xi > \xi^{\ast}$ (see Fig.~\ref{fig:bottlenecks}).
%, which we define as \begin{equation} \label{eq:regN}
%    \mathcal{N}:= \left\{ (\xi, \rho) \, : \, 0 %< \xi < 1, \, \rho_c^-(\xi) < \rho < %\rho_c^+(\xi) \right\}.
%\end{equation}
Since no transitions from $\xi =0$ to $\xi =1$ are possible through the region $\mathcal{N}$, 
it plays no role in the
construction of solutions of the boundary value problem. Since other folded singularities associated with local minima or maxima of $k$
and their associated canard solutions are confined to $\mathcal{N}$ these also play no role for boundary value problem.
\end{remark}

\begin{figure}[!ht]
    \centering
    \begin{overpic}[scale=0.7]{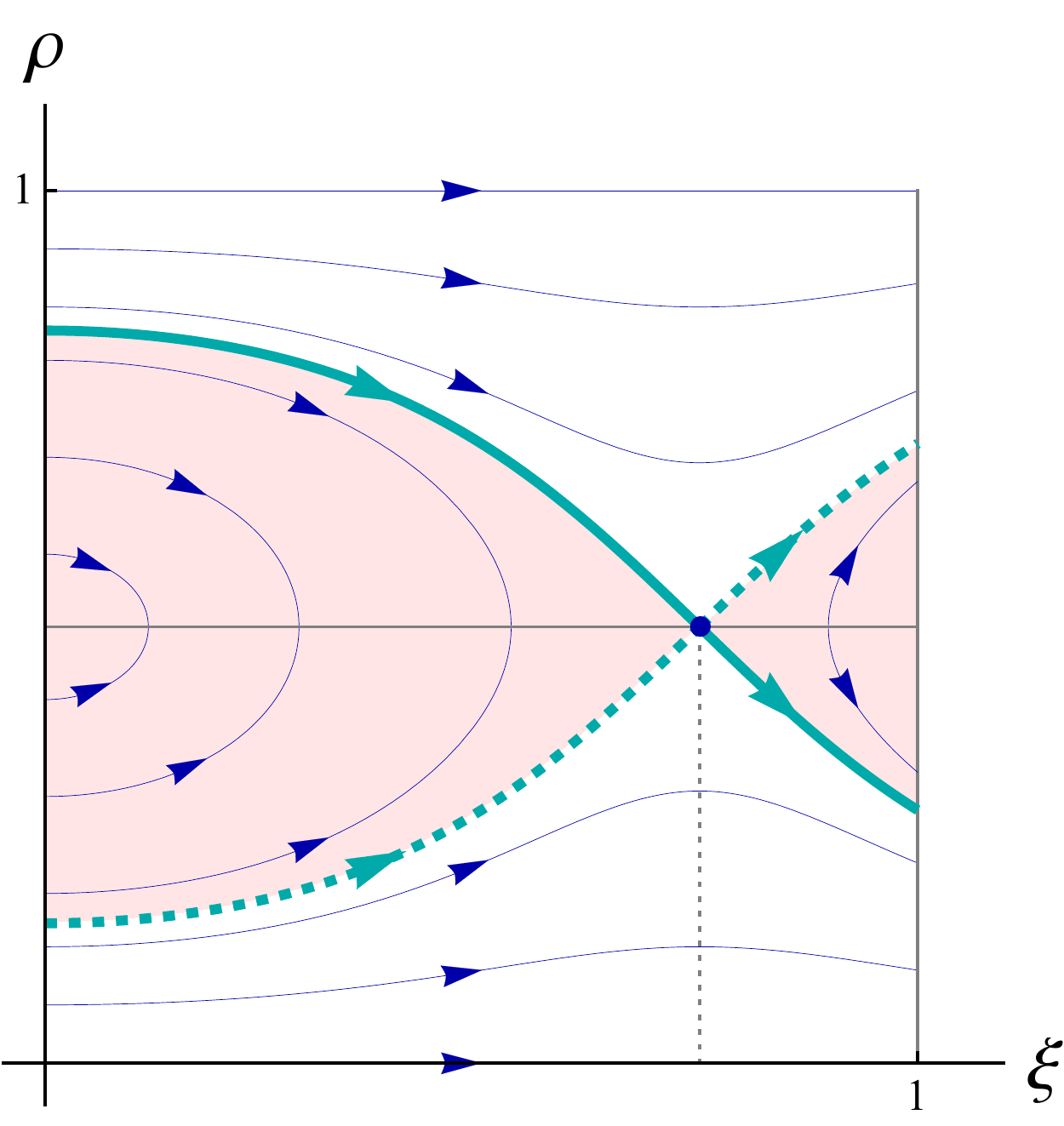}
  	\put(1,2){\small{$0$}}
    \put(61,2){$\xi^\ast$}
    \put(-2,17){$\rho_c^0$}
    \put(-8,70){$1-\rho_c^0$}
    \put(83,27){$1-\rho_c^1$}
    \put(83,61){$\rho_c^1$}
    \put(55,53){$S_c$}
    \put(55,33){$\tilde{S}_c$}
    \end{overpic}
    \caption{
    %\blu{TODO Caption needs to be adapted to Lemma 1 and "new" figure}.
    Illustration of the reduced flow associated to Equations \eqref{eq:sys}-\eqref{eq:bc} described in Lemma \ref{lem:reducedflow} for $k(\xi)=1+a\,\mathrm{cos}\left(\frac{2 \pi \xi}{b} \right)$ and $a=0.3$, $b=1.5$. The solid gray line indicates the line of fold points $F$ (see \eqref{eq:foldL}), whereas the blue dot corresponds to the canard point of folded saddle type $p^\ast$. The cyan curves correspond to the canard $S_c$ (solid line) and the (faux) canard $\tilde{S}_c$ (dashed line). The shaded red area represents the region $\mathcal{N}$ defined in Remark 3, which plays no role in the construction of solutions of the boundary value problem \eqref{eq:sys}-\eqref{eq:bc}.}
    \label{fig:bottlenecks}
\end{figure}
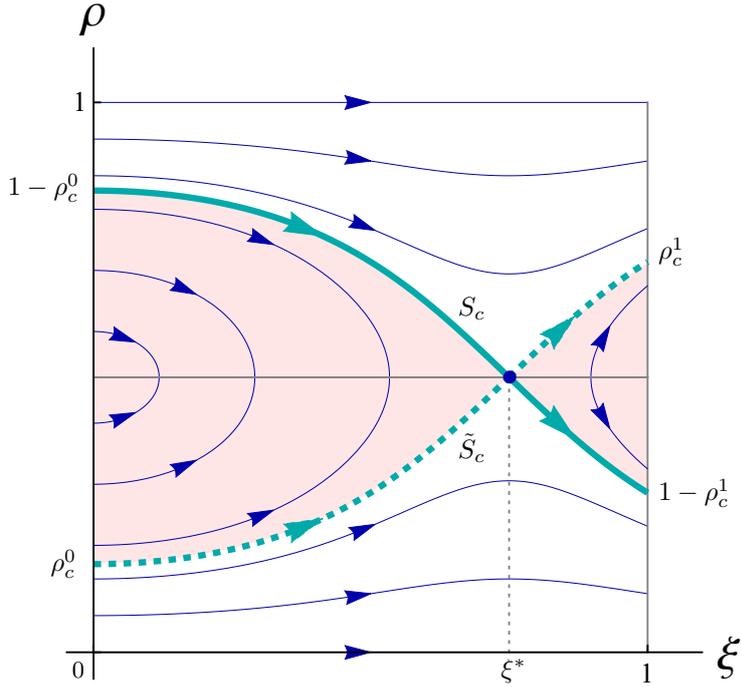

We now begin the construction of solutions of the boundary value problem \eqref{eq:sys}-\eqref{eq:bc} by combining
solutions of the reduced problem with solutions of the layer problem in such a way that the boundary conditions are satisfied.
Here it is important to keep in mind that solutions can jump from points on the repelling branch $\mathcal{C}_0^r$ of the critical
manifold to the attracting branch $\mathcal{C}_0^a$, but not vice versa.
In \cite{Iuorio_2022} we have constructed singular solutions  in the case of a closing channel using a shooting strategy: we evolved the manifold of boundary conditions at $\xi=0$ forward and checked whether it intersected the manifold of boundary conditions at $\xi=1$. This constructive procedure allowed to  identify the initial and final values of $\rho$ (namely $\rho_0$ and $\rho_1$) for $\varepsilon=0$. In the bottleneck scenario, however, the presence of a canard point lying in the interior of the spatial domain $[0,1]$ implies that 
singular orbits containing segments of the canards $S_c$ or $\tilde{S}_c$ can make slow transitions between the branches of the critical manifold. Most importantly, this allows transitions from the attracting branch back to the repelling branch. We will show
that this leads to the new type of transitional profiles, described in the introduction.
 Due to the special role of the canard point $p^{\ast}$ we modify the shooting strategy by evolving also the manifold of boundary conditions at $\xi=1$ (backwards) and checking the intersection with the forward evolution of the manifold of left boundary conditions at $\xi=\xi^\ast$, where the canard point $p^\ast$ lies.

%\section{New construction}
In the dynamical systems framework, boundary conditions \eqref{eq:bc} correspond to two lines in the $(j,\xi,\rho)$-space, satisfying $j=\alpha(1-\rho)$ at $\xi=0$ and $j=\beta \rho$ at $\xi=1$, respectively. However, due to the fast-slow structure, the set of admissible boundary conditions is restricted to  (see Figures \ref{fig:L_comb}-\ref{fig:R_comb})
\begin{subequations} \label{eq:LR}
\begin{align}
 \mathcal{L} &:= \left\{ \left( \alpha(1-s), \, 0, \, s \right) \ : \ \rho_\alpha \leq s \leq 1 \right\}, \\
 \mathcal{R} &:= \left\{ \left( \beta t, \, 1, \, t \right) \ : \ 0 \leq t \leq \rho_\beta \right\}.
\end{align}
\end{subequations}
Here
\begin{equation} \label{eq:rho_alpha}
 \rho_\alpha= \begin{cases} \alpha &\quad \text{if } \alpha \leq \frac12, \\ 1-\frac{1}{4 \alpha} &\quad \text{if } \alpha \geq \frac12, \end{cases}
\end{equation}
and
\begin{equation} \label{eq:rho_beta}
 \rho_\beta= \begin{cases} 1-\beta &\quad \text{if } \beta \leq \frac12, \\ \frac{1}{4\beta} &\quad \text{if } \beta \geq \frac12. \end{cases}
\end{equation}
The lower and upper bounds $\rho_\alpha$ and $\rho_\beta$ for the density $\rho$ are caused by the fast-slow structure of the flow: if we would consider a starting point $(\alpha(1-\rho), \, 0, \, \rho)$ with $0 \leq \rho < \rho_\alpha$, the orbit would be immediately repelled to infinity from $\mathcal{C}_0$, hence connecting to the boundary conditions at $\xi=1$ is impossible. Analogously, points satisfying $(\beta \rho, \, 1, \, \rho)$ with $\rho_\beta < \rho \leq 1$ cannot be endpoints of the singular orbits, since they are repelling for the layer problem.\\
Thus, the initial and final points of the singular orbits -- $p_0$ and $p_1$, respectively -- must satisfy
\begin{equation} \label{eq:p_0f}
 p_0 \in \mathcal{L}, \textrm{ and } p_1 \in \mathcal{R}.
\end{equation}
The manifold $\mathcal{L}$ intersects with $\mathcal{C}_0$ at $(0,0,1)$ and
 \begin{equation} \label{eq:p_l}
 l=(\alpha(1-\alpha),0,\alpha),
 \end{equation}
while $\mathcal{R}$ intersects with $\mathcal{C}_0$ at $(0,1,0)$ and
 \begin{equation} \label{eq:p_r}
 r=(\beta(1-\beta),1,1-\beta).
 \end{equation}
For $\varepsilon=0$, the variable $\xi$ evolves only on $\mathcal{C}_0$ according to the reduced flow \eqref{eq:k_redpb_xrho}. Therefore, in order for the singular solution to evolve from $\xi=0$ to $\xi=1$, we must connect $\mathcal{L}$ and $\mathcal{R}$ to $\mathcal{C}_0$. The points $l$ and $r$ already belong to $\mathcal{C}_0$. Other points on $\mathcal{L}$ and $\mathcal{R}$ can reach $\mathcal{C}_0$ using the layer problem \eqref{eq:laypb_k}. Tracking the evolution of $\mathcal{L}$ by means of the layer problem at $\xi=0$ until it reaches $\mathcal{C}_0$, and analogously the evolution of $\mathcal{R}$ backwards until the layer problem at $\xi=1$ intersects $\mathcal{C}_0$, yields two sets (shown in Figures \ref{fig:L_comb}-\ref{fig:R_comb}):
\begin{subequations} \label{eq:M010_i}
\begin{align}
 \mathcal{L}^+ &:= \left\{ \left( \alpha(1-s), \, 0, \, \rho^+(0,s) \right) \ : \ \rho_\alpha \leq s \leq 1 \right\},
 \label{eq:L+} \\
 \mathcal{R}^- &:= \left\{ \left( \beta t, \, 1, \,  \rho^-(1,t)\right) \ : \ 0 \leq t \leq \rho_\beta \right\}. %\frac12\left(1-\sqrt{1-4\beta t}\right)
 \label{eq:R-}
\end{align}
\end{subequations}
In the following, we use the symbol $\rho^+(0,s)$ to indicate the $\rho$-value (greater than or equal to $\frac12$) reached by the point $(\alpha(1-s), \, 0, \, s)$ after its transition from $\mathcal{L}$ to $\mathcal{C}_0$ by means of the layer problem. If the solution $(\xi,\rho)$ of the reduced flow \ref{eq:k_redpb_xrho} starting at $(0,\rho^+(0,s))$ reaches $\xi=\xi^\ast$, we denote its value of $\rho$ at $\xi=\xi^\ast$ by $\rho^+\left(\xi^\ast,s\right)$. In an analogous manner, we introduce the symbol $\rho^-(1,t)$ to indicate the $\rho$-value (less than or equal to $\frac12$) reached by the point $(\beta t, \, 1, \, t)$ after its transition from $\mathcal{R}$ to $\mathcal{C}_0$ by means of the layer problem. If the solution $(\xi,\rho)$ of the reduced flow starting at $(1,\rho^-(1,t))$ and flowing backwards reaches $\xi=\xi^\ast$, we denote its value of $\rho$ at $\xi=\xi^\ast$ by $\rho^-\left(\xi^\ast,t\right)$. \\
%We observe that, by symmetry, $\rho^-=-\rho^+$.\\
When $\alpha < \frac12$, the reduced flow can either start on $\mathcal{L}^+$ or at $l$, while for $\alpha > \frac12$ it must start on $\mathcal{L}^+$. Analogously, when $\beta<\frac12$, the reduced flow can either end on $\mathcal{R}^-$ or at $r$, while for $\beta > \frac12$ it must end on $\mathcal{R}^-$.
 \begin{figure}[H] 
 \centering
 \begin{minipage}{.4\textwidth}
  \centering
  \def\svgwidth{.8\textwidth}
  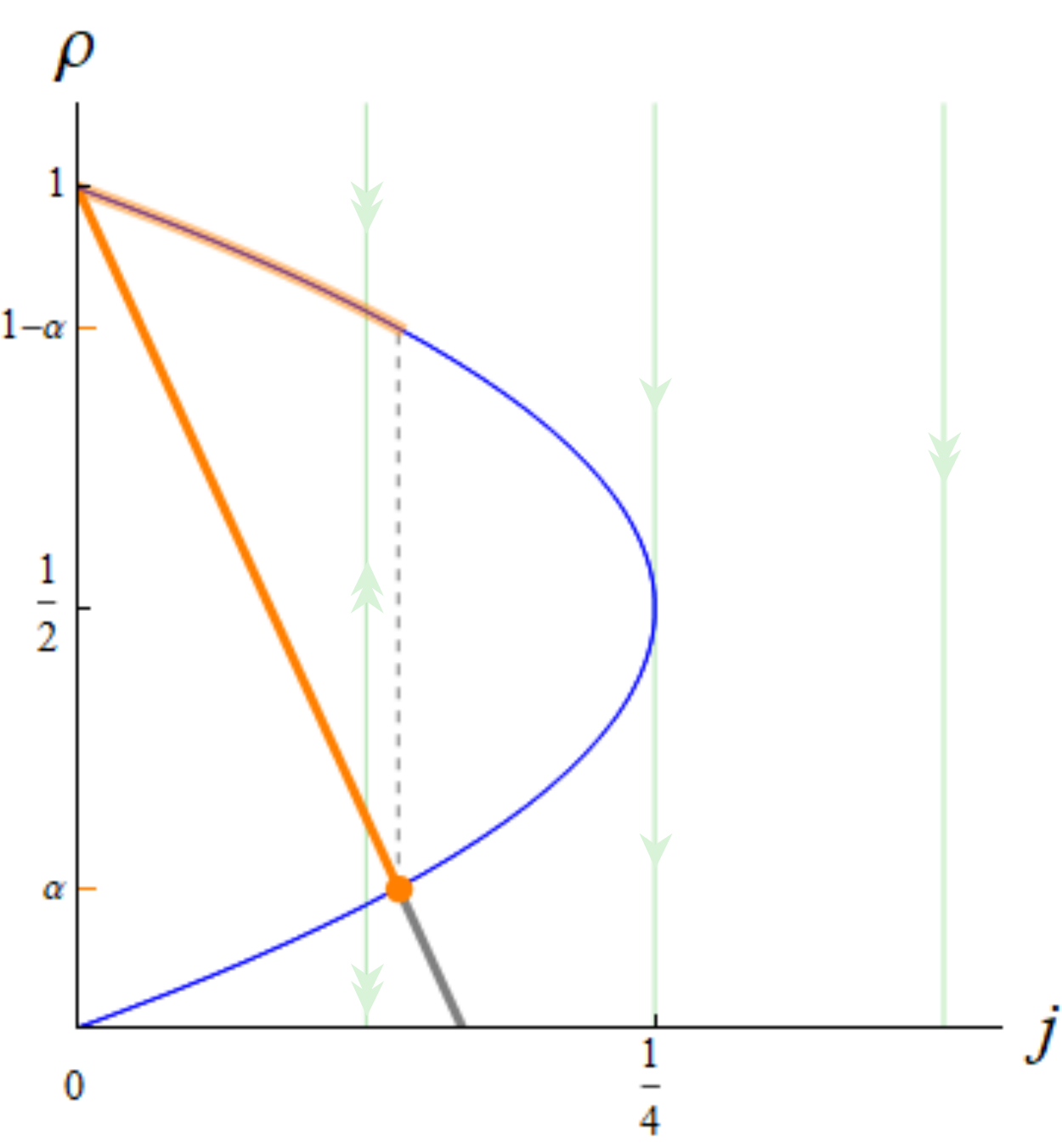\\
  (a)
 \end{minipage}
 \begin{minipage}{.4\textwidth}
  \centering
  \def\svgwidth{.8\textwidth}
  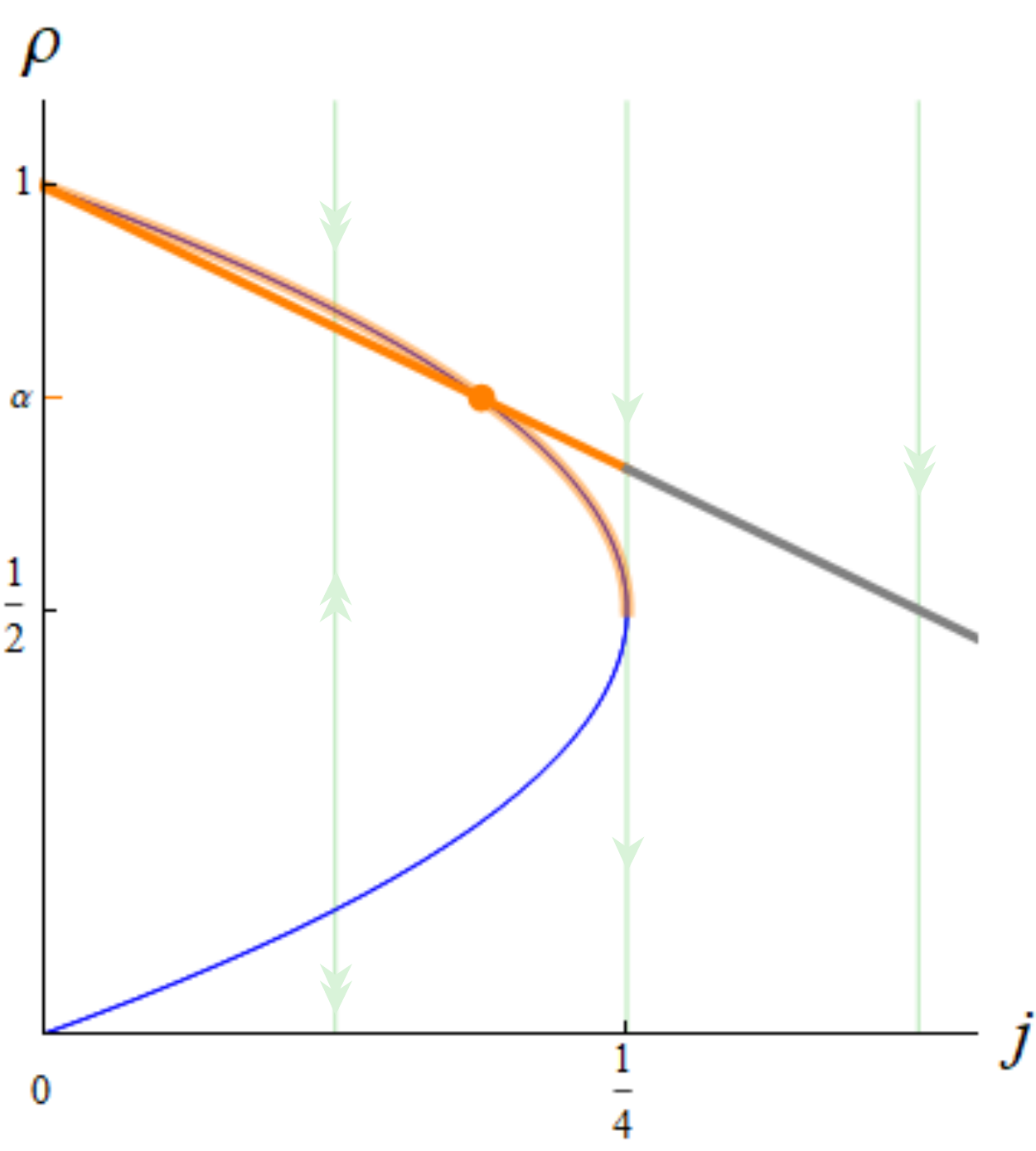\\
  (b)
 \end{minipage}
 \caption{Schematic representation of $\mathcal{L}$ (orange line) and $\mathcal{L}^+$ (orange curve) for (a) $0 <\alpha < \frac12$ and (b) $\frac12 <\alpha < 1$. The orange dot corresponds to $l$, the blue curve represents $\mathcal{C}_0$, and the green lines correspond to the orbits of the layer problem.}
 \label{fig:L_comb}
 \end{figure}
 
 \begin{figure}[H] 
 \centering
 \begin{minipage}{.4\textwidth}
  \centering
  \def\svgwidth{.8\textwidth}
  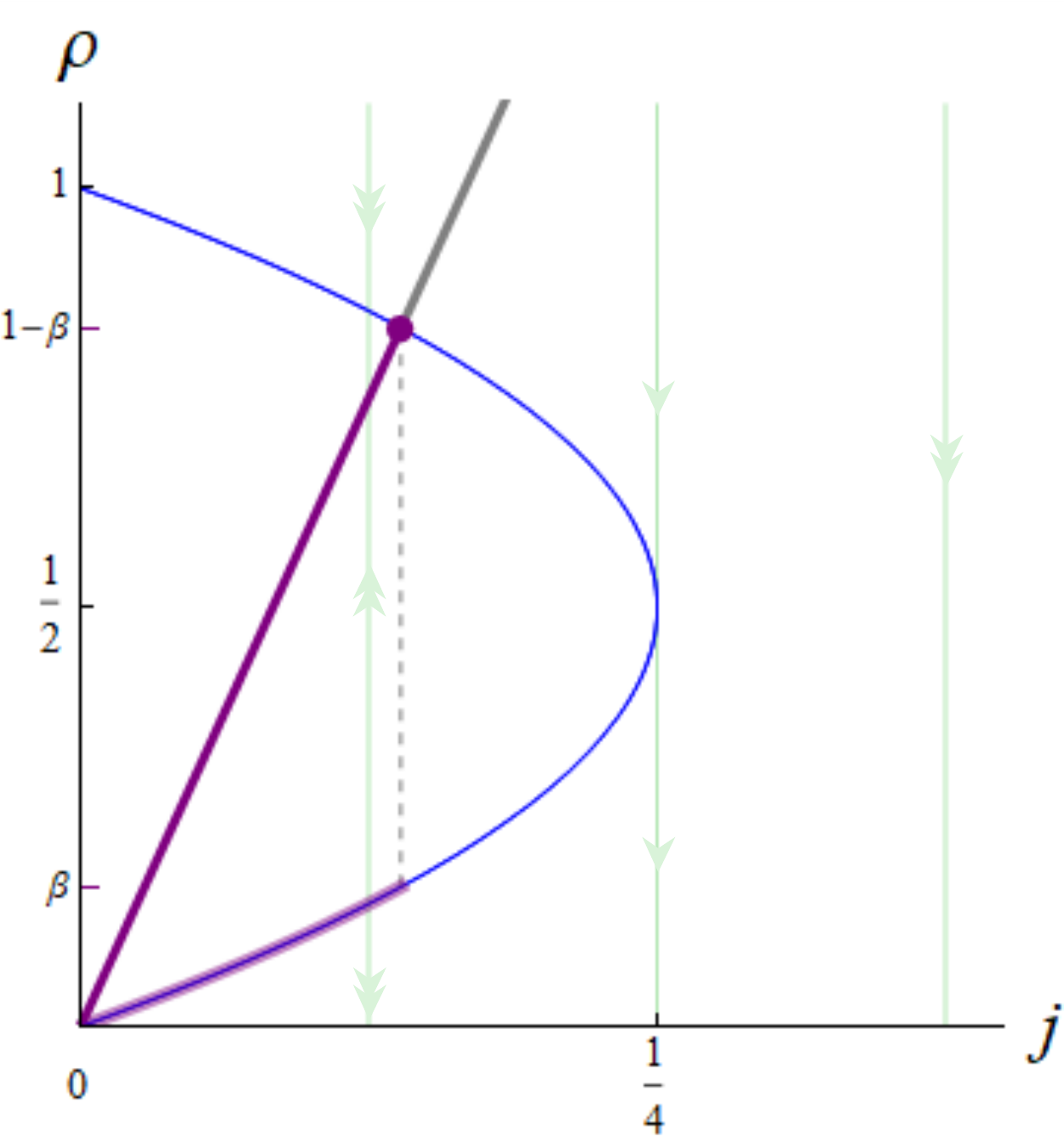\\
  (a)
 \end{minipage}
 \begin{minipage}{.4\textwidth}
  \centering
  \def\svgwidth{.8\textwidth}
  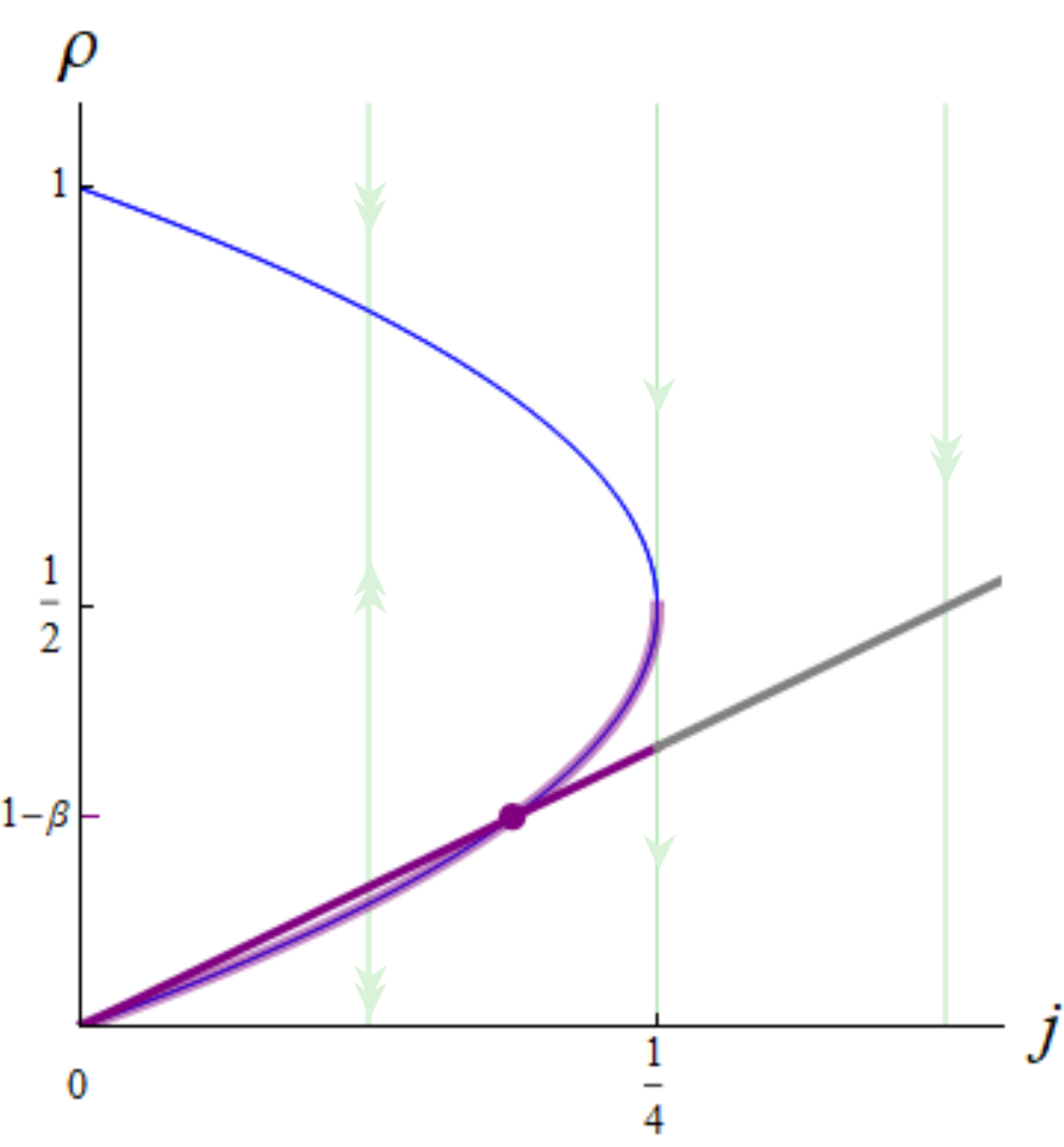\\
  (b)
 \end{minipage}
 \caption{Schematic representation of $\mathcal{R}$ (purple line) and $\mathcal{R}^-$ (purple curve) for (a) $0 < \beta < \frac12$ and (b) $\frac12 <\beta < 1$. The purple dot corresponds to $r$, the blue curve represents $\mathcal{C}_0$, and the green lines correspond to the orbits of the layer problem.}
 \label{fig:R_comb}
 \end{figure}

Based on this geometric interpretation of the boundary conditions, we proceed with the construction of the singular orbits by connecting $\mathcal{L}^+ \cup l$ and  $\mathcal{R}^- \cup r$ by means of the reduced flow \eqref{eq:k_redpb_xrho} on $\mathcal{C}_0$. In doing so, we first let $\mathcal{L}^+$ in \eqref{eq:L+} flow forward and $\mathcal{R}^-$ in \eqref{eq:R-} flow backwards by means of the reduced flow until $\xi=\xi^\ast$: we call the corresponding sets $\mathcal{L}_{\xi^\ast}^+$ and $\mathcal{R}_{\xi^\ast}^-$, respectively.
\begin{subequations} \label{eq:LR_12}
\begin{align}
 \mathcal{L}_{\xi^\ast}^+ &:= \left\{ \left( \rho^+ \left(\xi^\ast,s\right)\left(1-\rho^+\left(\xi^\ast,s\right)\right),\, \xi^\ast,\, \rho^+\left(\xi^\ast,s\right)  \right) \ : \ \rho_\alpha \leq s \leq 1 \right\},
 \label{eq:L+_rast} \\
 \mathcal{R}_{\xi^\ast}^- &:= \left\{ \left( \rho^- \left(\xi^\ast,t\right)\left(1-\rho^-\left(\xi^\ast,t\right)\right),\, \xi^\ast,\, \rho^-\left(\xi^\ast,t\right)  \right) \ : \ 0 \leq t \leq \rho_\beta \right\}.
 \label{eq:R-_rast}
\end{align}
\end{subequations}
If $\alpha \geq \frac12$ then $l \in \mathcal{L}^+$, and the evolution of $l$ by means of the reduced flow is already included in $\mathcal{L}_{\xi^\ast}^+$. If $\alpha < \frac12$ then $l \notin \mathcal{L}^+$, and therefore the corresponding point at $\xi=\xi^\ast$ must be defined separately as
\begin{equation} \label{eq:l_12}
% \begin{aligned}
 l_{\xi^\ast} := \left\{ \left( \rho^+\left(\xi^\ast,\alpha\right)\left(1-\rho^+\left(\xi^\ast,\alpha\right)\right), \, \xi^\ast,\, 1-\rho^+\left(\xi^\ast,\alpha\right) \right) \right\}.
% \end{aligned}
\end{equation}
Analogously, if $\beta \geq \frac12$ then $r \in \mathcal{R}^-$, and the backwards evolution of $r$ by means of the reduced flow is already included in $\mathcal{R}_{\xi^\ast}^-$. If $\beta < \frac12$, however, $r \notin \mathcal{R}^-$, and therefore the corresponding point at $\xi=\frac12$ must be defined separately as
\begin{equation} \label{eq:r_12}
% \begin{aligned}
 r_{\xi^\ast} := \left\{ \left(\rho^-\left(\xi^\ast,1-\beta\right)\left(1-\rho^-\left(\xi^\ast,1-\beta\right)\right), \, \xi^\ast,\, 1-\rho^-\left(\xi^\ast,1-\beta\right)\right) \right\}.
% \end{aligned}
\end{equation}
%Due to the structure of the reduced flow, the points $l_{\xi^\ast}$ and $r_{\xi^\ast}$ do not exist for all values of $\alpha < \frac12$ and $\beta < \frac12$, respectively (see Remark \ref{rem:lr_12}).\\
We note that the point $l_{\xi^\ast}$ exists if and only if $\alpha \leq \rho_c^0$; analogously, the point $r_{\xi^\ast}$ exists if and only if $\beta \leq 1-\rho_c^1$ (see Remark \ref{rem:regN}). \\
A singular orbit is then given by matching the slow and fast pieces obtained by investigating the reduced and layer problems, respectively. More specifically, a singular orbit exists if and only if the intersection between the sets $\mathcal{L}_{\xi^\ast}^+ \cup l_{\xi^\ast}$ and $\mathcal{R}_{\xi^\ast}^- \cup r_{\xi^\ast}$ is non-empty, and it is unique if this intersection consists of one point.\\
In addition to the canards $S_c$ and $\tilde{S}_c$ introduced above, our analysis of the existence and structure of singular orbits is based on four special orbits $S_{\alpha}$, $S_{\beta}$, $\tilde{S}_{\alpha}$, $\tilde{S}_{\beta}$ of the reduced flow (see Figure \ref{fig:sf_sp}):
\begin{itemize}
 %\item \redd{The orbit $S_c$ -- i.e. the \emph{faux canard} orbit} is defined as the one starting at $\xi=0$ at a $\rho$-value below $\frac12$ and satisfying $\rho=\frac12$ at $\xi=\xi^\ast$. We refer to the corresponding initial value at $\xi=0$ by $\rho_c^0$ and the corresponding final value at $\xi=1$ by $\rho_c^1$.
 \item The orbit $S_{\alpha}$, defined for each $\alpha \in (0,1)$, is the one starting at $\rho=\alpha$ at $\xi=0$. For $\alpha < \rho_c^0$ or $\alpha >  1-\rho_c^0$, the corresponding final value of $\rho$ at $\xi=1$ is denoted by $\rho^\ast(\alpha)$. For $\rho_c^0 < \alpha < 1-\rho_c^0$, $S_{\alpha}$ ends on the fold line $F$ and hence $S_{\alpha} \subset \mathcal{N}$. For $\alpha=\rho_c^0$ or $\alpha=1-\rho_c^0$, $S_{\alpha}$ ends on the canard point $p^\ast$ at $\xi = \xi^\ast$, and its continuation for $\xi \in [\xi^\ast,1]$ is therefore not uniquely defined.
 \item The orbit $S_{\beta}$, defined for each $\beta \in (0,1)$, is the one ending at $\rho=1-\beta$ at $\xi=1$. For $\beta < 1-\rho_c^1$ or $\beta > \rho_c^1$, the corresponding initial value of $\rho$ at $\xi=0$ is denoted by $\rho_\ast(\beta)$. For $1-\rho_c^1 < \beta < \rho_c^1$, $S_{\beta}$ ends on the fold line $F$ and hence $S_{\beta} \subset \mathcal{N}$. For $\beta=1-\rho_c^1$ or $\beta=\rho_c^1$, $S_{\beta}$ ends on the canard point $p^\ast$ at $\xi = \xi^\ast$ backward in $\xi$, and its continuation for $\xi \in [0, \xi^\ast]$ is therefore not uniquely defined.
 %When $\beta = \frac12$, we define $S_{\beta} = \tilde{S}_c$.
 \item For $i=\alpha,\beta$, we define $\tilde{S}_i$ as the reflection of the orbit $S_i$ with respect to $\rho=\frac12$.
\end{itemize}
%\textcolor{blue}{If $\alpha=\rho_c^0$, then $S_{\alpha}=S_c$ and $\tilde{S}_{\alpha}=\tilde{S}_c$. Moreover, if $\beta=\rho^\ast(\alpha)$, then $S_{\beta}=\tilde{S}_{\alpha}$ and $\tilde{S}_{\beta}=S_{\alpha}$.\\
%\noindent}

Depending on the values of $\alpha$ and $\beta$, one of the orbits $S_i$, $\tilde{S}_i$, $i=c, \alpha, \beta$, corresponds to the slow part of the singular orbits we will construct.\\
Changing $\alpha$ and $\beta$ influences the orbits $S_{\alpha}$, $\tilde{S}_{\alpha}$ and $S_{\beta}$, $\tilde{S}_{\beta}$. We will show in the following that the $\alpha$, $\beta$ dependent mutual position of these orbits determines the type of singular solution of the boundary value problem.

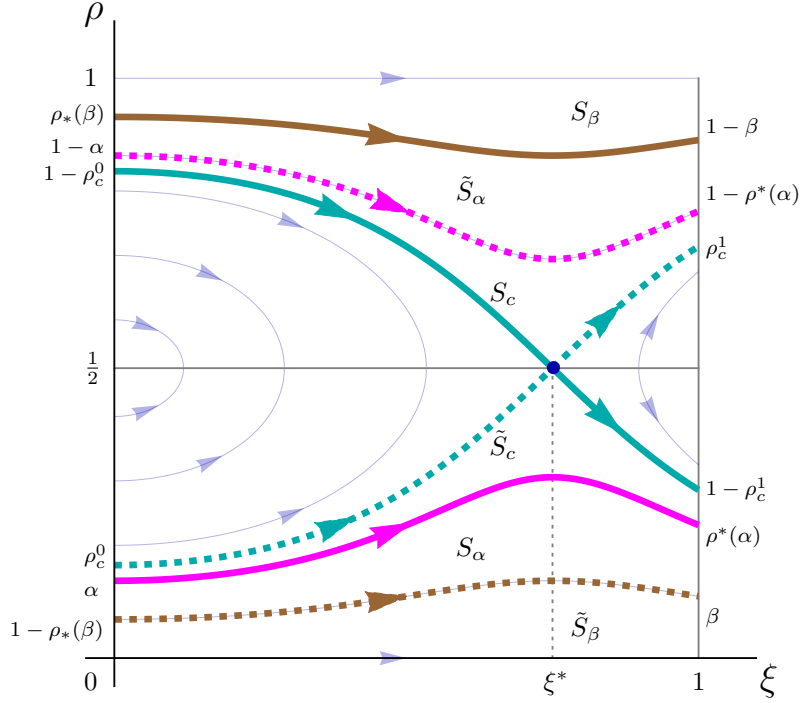
\begin{figure}[!ht]
  \centering
  	\begin{overpic}[scale=0.7]{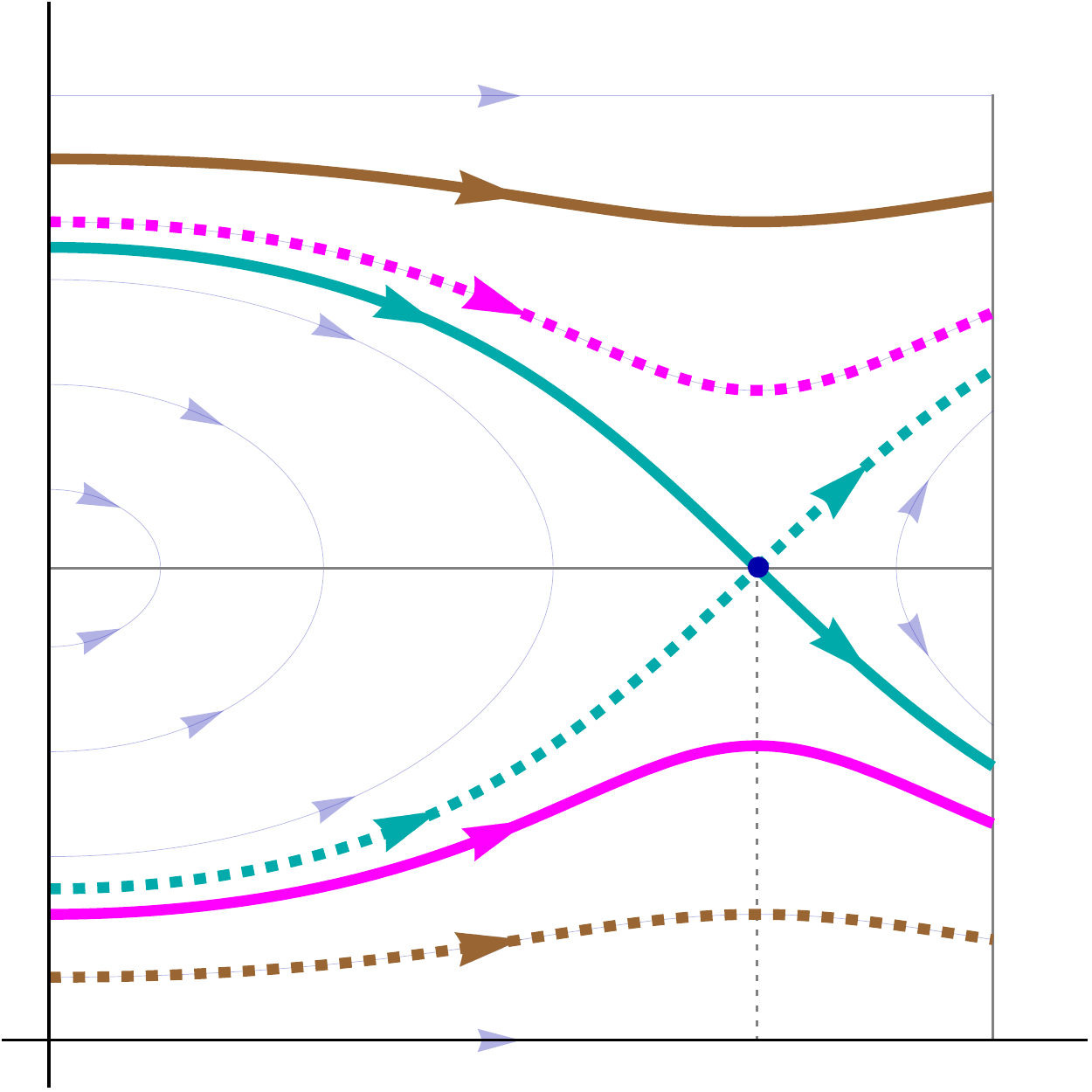}
  	\put(0,0){$0$}
  	\put(68,0){$\xi^\ast$}
  	\put(90,0){$1$}
  	\put(100,0){\Large$\xi$}
	\put(-11,8){\footnotesize $1-\rho_\ast(\beta)$}
    \put(0,14){\footnotesize $\alpha$}
    \put(0,19){\footnotesize $\rho_c^0$}
    \put(0,47){$\frac12$}
    \put(-6,76){\footnotesize $1-\rho_c^0$}
    \put(-5,80){\footnotesize $1-\alpha$}
    \put(-5,85){\footnotesize $\rho_\ast(\beta)$}
    \put(0,90){$1$}
    \put(0,100){\Large$\rho$}
    \put(92,10){\footnotesize $\beta$}
    \put(92,22){\footnotesize $\rho^\ast(\alpha)$}
    \put(92,29){\footnotesize $1-\rho_c^1$}
    \put(92,65){\footnotesize $\rho_c^1$}
    \put(92,73){\footnotesize $1-\rho^\ast(\alpha)$}
    \put(92,83){\footnotesize $1-\beta$}
    \put(60,35){$\tilde{S}_c$}
    \put(60,58){$S_c$}
    \put(55,20){$S_{\alpha}$}
    \put(55,73){$\tilde{S}_{\alpha}$}
    \put(72,85){$S_{\beta}$}
    \put(72,8){$\tilde{S}_{\beta}$}
	\end{overpic}
  \caption{Schematic illustration of the special orbits $S_c$, $S_{\alpha}$, $S_{\beta}$, $\tilde{S}_c$, $\tilde{S}_{\alpha}$, $\tilde{S}_{\beta}$ in $(\xi,\rho)$-space in the case $0 < \alpha < \rho_c^0$ and $\rho_c^1 < \beta < 1$. Here $k(\xi)=1+a\,\mathrm{cos}\left(\frac{2 \pi \xi}{b} \right)$ with $a=0.3$, $b=1.5$. The orbit $S_c$ (solid cyan curve) connects $(0,1-\rho_c^0)$ and $(1,1-\rho_c^1)$, the orbit $S_{\alpha}$ (solid magenta curve) connects $(0,\alpha)$ and $(1,\rho^\ast(\alpha))$, and the orbit $S_{\beta}$ (solid brown curve) connects $(0,\rho_\ast(\beta))$ and $(1,1-\beta)$ with $\alpha,\beta < \frac12$ as in \eqref{eq:sp_rho}. The dashed curves represent the orbits $\tilde{S}_c$ (cyan), $\tilde{S}_{\alpha}$ (magenta), $\tilde{S}_{\beta}$ (brown), which are symmetric to the corresponding solid ones $S_c$, $S_{\alpha}$, $S_{\beta}$ with respect to $\rho=\frac12$.  We note that the orbits $S_{\alpha}$, $\tilde{S}_{\alpha}$ and $S_{\beta}$, $\tilde{S}_{\beta}$ switch position in the above diagram as $\alpha$, $\beta > \frac12$.}
    \label{fig:sf_sp}
\end{figure}
By using the conserved quantity \eqref{Eq:H}
the respective values of $\rho^\ast(\alpha)$ and  $\rho_\ast(\beta)$
can be computed explicitly:
\begin{subequations} \label{eq:sp_rho}
\begin{align}
 \rho^\ast(\alpha) &:= 
              \left\{
                \begin{array}{ll}
                  \frac12 \left( 1 - \sqrt{1-4\alpha(1-\alpha) \frac{k(0)}{k(1)}} \right) \text{ if } \alpha < \rho_c^0, \\
                  \frac12 \left( 1 + \sqrt{1-4\alpha(1-\alpha) \frac{k(0)}{k(1)}} \right) \text{ if } \alpha > 1 - \rho_c^0,
                \end{array}
              \right.\\
 \rho_\ast(\beta) &:=
              \left\{
                \begin{array}{ll}
                  \frac12 \left( 1 + \sqrt{1-4\beta(1-\beta) \frac{k(1)}{k(0)}} \right) \text{ if } \beta < 1-\rho_c^1, \\
                  \frac12 \left( 1 - \sqrt{1-4\beta(1-\beta) \frac{k(1)}{k(0)}} \right) \text{ if } \beta > \rho_c^1. \\
                \end{array}
              \right.
%  \frac12 \left( 1 - \sqrt{1-4\alpha(1-\alpha) \frac{k(0)}{k(1)}} \right), \\
%  \rho_\ast(\beta) &:= \frac12 \left( 1 + \sqrt{1-4\beta(1-\beta) \frac{k(1)}{k(0)}} \right).
\end{align}
\end{subequations}
Note that $\alpha=1-\rho_\ast(\beta)$ is equivalent to $\beta=\rho^\ast(\alpha)$.
 
Based on this, we divide the $(\alpha,\beta)$-parameter space into eight regions $\mathcal{G}_i$, $i=1,\dots,8$ defined via the following curves $\gamma_{ij}$ (here the indices refer to the adjacent regions):
\begin{subequations}\label{eq:g_bound}
 \begin{align}
  \gamma_{12} &:= \left\{ (\alpha,\beta) \ : \;  0< \alpha  \leq \rho_c^0, \, \beta = 1-\rho^\ast(\alpha) \right\},\\
  \gamma_{13} &:= \left\{ (\alpha,\beta) \ : \; \alpha = \rho_c^0, \, 1-\rho_c^1 \leq \beta \leq \rho_c^1 \right\},\\
  \gamma_{17} &:= \left\{ (\alpha,\beta) \ : \; \alpha = 1-\rho_\ast(\beta), \, 0 < \beta \leq 1-\rho_c^1 \right\},\\
  \gamma_{24} &:= \left\{ (\alpha,\beta) \ : \; \alpha = \rho_c^0, \, \rho_c^1 \leq \beta < 1 \right\},\\
  \gamma_{34} &:= \left\{ (\alpha,\beta) \ : \; \rho_c^0 \leq \alpha \leq 1-\rho_c^0, \, \beta = \rho_c^1 \right\},\\
  \gamma_{35} &:= \left\{ (\alpha,\beta) \ : \; \alpha = 1-\rho_c^0, \, 1-\rho_c^1 \leq \beta \leq \rho_c^1 \right\},\\
  \gamma_{37} &:= \left\{ (\alpha,\beta) \ : \; \rho_c^0 \leq \alpha \leq 1-\rho_c^0, \, \beta = 1-\rho_c^1 \right\},\\
  \gamma_{46} &:= \left\{ (\alpha,\beta) \ : \; \alpha = 1-\rho_c^0, \, \rho_c^1 \leq \beta < 1 \right\},\\
  \gamma_{56} &:= \left\{ (\alpha,\beta) \ : \; 1-\rho_c^0 \leq \alpha < 1, \, \beta = \rho_c^1 \right\},\\
  \gamma_{58} &:= \left\{ (\alpha,\beta) \ : \; 1-\rho_c^0 \leq \alpha < 1, \, \beta = 1-\rho_c^1 \right\},\\
  \gamma_{78} &:= \left\{ (\alpha,\beta) \ : \; \alpha = \rho_\ast(\beta), \, 0 < \beta \leq 1-\rho_c^1 \right\}.
  \end{align}
 \end{subequations}
The above curves correspond to situations where some of the orbits $S_i$, $\tilde{S}_i$, $i=1,2,3$ defined above coincide. In particular:
\begin{itemize}
    \item for $(\alpha,\beta) \in \gamma_{12}$, we have $S_{\alpha}=S_{\beta}$ (lying in $\mathcal{C}_0^r$);
    \item for $(\alpha,\beta) \in \gamma_{13} \cup \gamma_{24}$, we have $\tilde{S}_c=S_{\alpha}$ for $\xi \in [0, \xi^\ast]$ (i.e. up to the canard point $p^\ast$);
    \item for $(\alpha,\beta) \in \gamma_{17}$, we have $S_{\alpha}=\tilde{S}_{\beta}$;
    \item for $(\alpha,\beta) \in \gamma_{34} \cup \gamma_{56}$, we have $S_c=S_{\beta}$ for $\xi \in [\xi^\ast, 1]$ (i.e. up to the canard point $p^\ast$);
    %the orbits $\tilde{S}_c$ and $S_{\alpha}$ coincide;
    \item for $(\alpha,\beta) \in \gamma_{35} \cup \gamma_{46}$, we have $S_c=S_{\alpha}$ for $\xi \in [0, \xi^\ast]$ (i.e. up to the canard point $p^\ast$);
    \item for $(\alpha,\beta) \in \gamma_{37} \cup \gamma_{58}$, we have $\tilde{S}_c=S_{\beta}$ for $\xi \in [\xi^\ast, 1]$ (i.e. up to the canard point $p^\ast$);
    \item for $(\alpha,\beta) \in \gamma_{78}$, we have $S_{\alpha}=S_{\beta}$ (lying in $\mathcal{C}_0^a$).
    %the orbits $S_{\alpha}$ and $S_{\beta}$ coincide (and lie in $\mathcal{C}_0^a$).
\end{itemize}
\begin{remark}
Whenever two orbits coincide, their symmetric reflections with respect to $\rho=\frac12$ coincide as well.
\end{remark}

The eleven curves in \eqref{eq:g_bound} split $(0,1)^2$ into $8$ regions $\mathcal{G}_i$, $i=1,\dots,8$ (shown in Figure \ref{fig:bifdiag}):
\begin{subequations} \label{eq:regions}
\begin{align}
    \mathcal{G}_1 &:= \left\{ (\alpha,\beta) \ : \; 0 < \alpha < \rho_c^0, \, \rho^\ast(\alpha) < \beta < 1-\rho^\ast(\alpha) \right\} \\
    \mathcal{G}_2 &:= \left\{ (\alpha,\beta) \ : \; 0 < \alpha < \rho_c^0, \, 1-\rho^\ast(\alpha) < \beta < 1 \right\}, \\
    \mathcal{G}_3 &:= \left\{ (\alpha,\beta) \ : \; \rho_c^0 < \alpha < 1-\rho_c^0, \, 1-\rho_c^1 < \beta < \rho_c^1 \right\}, \\
    \mathcal{G}_4 &:= \left\{ (\alpha,\beta) \ : \; \rho_c^0 < \alpha < 1-\rho_c^0, \, \rho_c^1 < \beta < 1 \right\}, \\
    \mathcal{G}_5 &:= \left\{ (\alpha,\beta) \ : \; 1-\rho_c^0 < \alpha < 1, \, 1-\rho_c^1 < \beta < \rho_c^1 \right\}, \\
    \mathcal{G}_6 &:= \left\{ (\alpha,\beta) \ : \; 1-\rho_c^0 < \alpha < 1, \, \rho_c^1 < \beta < 1 \right\}, \\
    \mathcal{G}_7 &:= \left\{ (\alpha,\beta) \ : \; 1-\rho_\ast(\beta) < \alpha < \rho_\ast(\beta), \, 0 < \beta < 1-\rho_c^1  \right\}, \\
    \mathcal{G}_8 &:= \left\{ (\alpha,\beta) \ : \; \rho_\ast(\beta) < \alpha < 1, \, 0 < \beta < 1-\rho_c^1  \right\}.
\end{align}
\end{subequations} 

In short terms, moving from one region to the other in the $(\alpha,\beta)$ -parameter space leads to a corresponding change in the structure of the singular solutions.

\begin{figure}[!ht]
    \centering
    \includegraphics[scale=1]{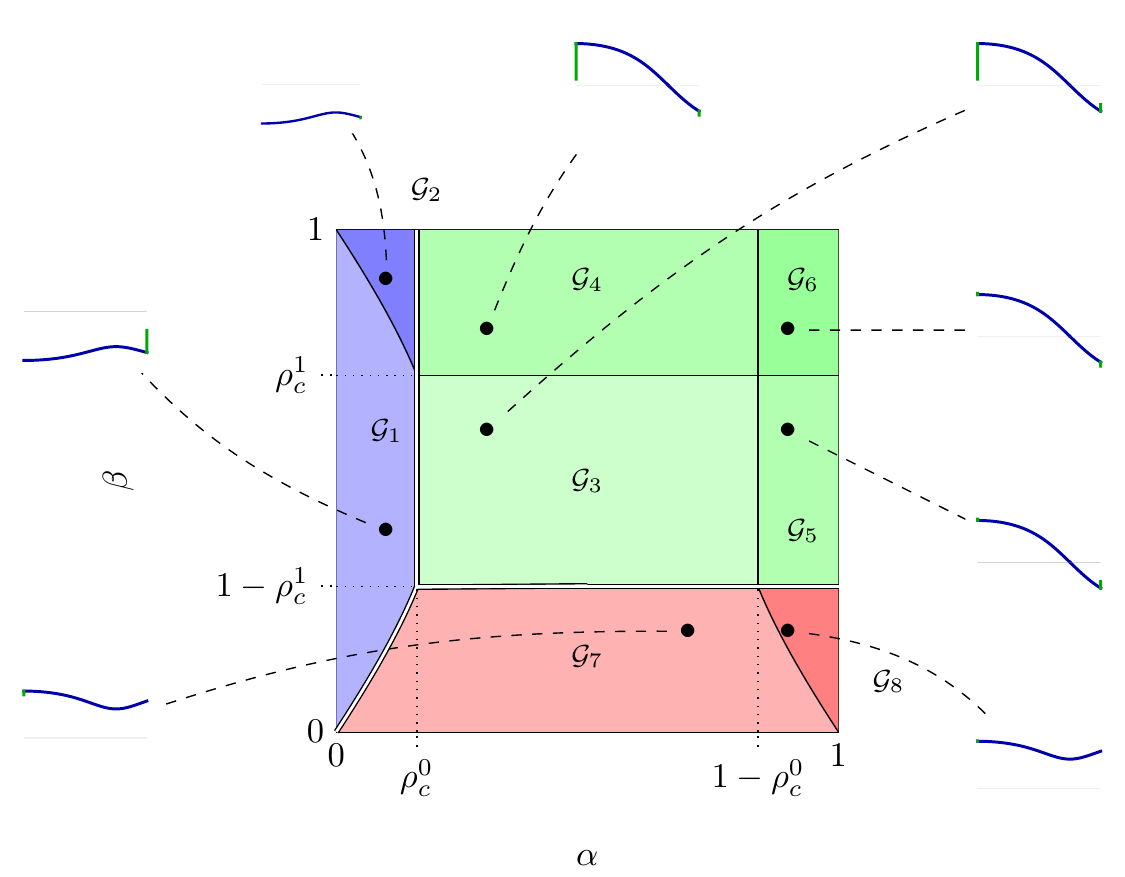}
    \caption{Representation of the $(\alpha,\beta)$ bifurcation diagram for $\varepsilon=0$. Here \mbox{$k(\xi)=1+a\,\mathrm{cos}\left(\frac{2 \pi \xi}{b} \right)$} with $a=0.3$, $b=1.5$. Red regions correspond to high density, blue regions to low density, and green regions to transitions from high to low density regimes. In the insets, the density $\rho$ is shown as a function of $\xi$. The blue parts correspond to solutions of the reduced problem \eqref{eq:k_redpb_xrho}, whereas the green parts indicate boundary layers. The gray line represents $\rho=\frac12$}
    \label{fig:bifdiag}
\end{figure}

We will show (in Proposition \ref{prop:singsol}) that within each of those region the structure of the singular solutions is the same. Note that our construction of singular solutions works also on all the boundary curves defined in \eqref{eq:g_bound} except for $\gamma_{17}$, $\gamma_{13} \cup \gamma_{24}$, and $\gamma_{37} \cup \gamma_{58}$,  where singular solutions are not unique (see Remark \ref{rem:fmam}).\\
\noindent
To this aim, we introduce the following eight types of singular solutions (see Figure \ref{fig:singsol1}-\ref{fig:singsol2}): %\ref{fig:bd_sing}(a)
\begin{description}
\item{Type 1.} Singular solutions which start on $\mathcal{C}_0^r$ at $\xi=0$, follow the reduced flow on $\mathcal{C}_0^r$ (where $\rho$ increases), and have a layer at $\xi=1$ in which $\rho$ increases.
\item{Type 2.} Singular solutions which start on $\mathcal{C}_0^r$ at $\xi=0$, follow the reduced flow on $\mathcal{C}_0^r$ (where $\rho$ increases), and have a layer at $\xi=1$ in which $\rho$ decreases.
\item{Type 3.} Singular solutions which have a layer at $\xi=0$ in which $\rho$ increases, follow the reduced flow on $\mathcal{C}_0$ (where $\rho$ decreases) passing through the point $p^\ast$, and have another layer at $\xi=1$ in which $\rho$ increases.
\item{Type 4.} Singular solutions which have a layer at $\xi=0$ in which $\rho$ increases, follow the reduced flow on $\mathcal{C}_0$ (where $\rho$ decreases) passing through the point $p^\ast$, and have another layer at $\xi=1$ in which $\rho$ decreases.
\item{Type 5.} Singular solutions which have a layer at $\xi=0$ in which $\rho$ decreases, follow the reduced flow on $\mathcal{C}_0$ (where $\rho$ decreases) passing through the point $p^\ast$, and have another layer at $\xi=1$ in which $\rho$ increases.
\item{Type 6.} Singular solutions which have a layer at $\xi=0$ in which $\rho$ decreases, follow the reduced flow on $\mathcal{C}_0$ (where $\rho$ decreases) passing through the point $p^\ast$, and have another layer at $\xi=1$ in which $\rho$ decreases.
\item{Type 7.} Singular solutions which have a layer at $\xi=0$ in which $\rho$ increases, and follow the reduced flow on $\mathcal{C}_0^a$ (where $\rho$ decreases).
\item{Type 8.} Singular solutions which have a layer at $\xi=0$ in which $\rho$ decreases and follow the reduced flow on $\mathcal{C}_0^a$ (where $\rho$ decreases).
\end{description}
 
\noindent More details about the construction and structure of these singular orbits are given in the proof of the following proposition.
\begin{proposition} \label{prop:singsol}
 Let $k \in C^2([0,1])$ be a positive function satisfying Assumption \eqref{k-assumption}.
 %which admits one global nondegenerate minimum at $x=x^\ast \in (0,1)$ and satisfies $\dt{k}(x) \leq 0$ for $x < x^\ast$, $\dt{k}(x) \geq 0$ for $x > x^\ast$
 Then for each $(\alpha,\beta) \in \mathcal{G}_i$, $i=1, \ldots, 8$  there exists a unique singular solution $\Gamma^i$ of type $i$ to \eqref{eq:reduced equation} composed of segments of orbits of the layer problem \eqref{eq:laypb_k} and the reduced problem \eqref{eq:k_redpb_xrho} satisfying the boundary conditions.
\end{proposition}
\begin{proof}
 The proof is based on the shooting technique outlined above. Technically speaking, we show that the intersection of the sets $\mathcal{L}_{\xi^\ast}^+ \cup l_{\xi^\ast}$ in \eqref{eq:L+_rast}-\eqref{eq:l_12} and $\mathcal{R}_{\xi^\ast}^- \cup r_{\xi^\ast}$ in \eqref{eq:R-_rast}-\eqref{eq:r_12} is non-empty, and in particular consists of one point. This gives us the unique values of $\rho_0,\,\rho_1$ for which a singular orbit exists depending on $\alpha$ and $\beta$, which in turn allows us to identify the eight types of singular solutions corresponding to the eight regions defined in \eqref{eq:regions}.
While we claim the existence of singular solutions only in the open regions $\Gamma^i$, $i=1,\ldots,8$ we also comment on the singular configurations
where $(\alpha,\beta)$ lies on the curves $\gamma_{ij}$ from (\ref{eq:g_bound}).\\
In principle there are four possible ways for the intersection between $\mathcal{L}_{\xi^\ast}^+ \cup l_{\xi^\ast}$ and $\mathcal{R}_{\xi^\ast}^- \cup r_{\xi^\ast}$ to occur; one of these defines four possible profiles corresponding to four regions in $(\alpha,\beta)$-parameter space, two of these lead to two possible profiles corresponding to two regions in $(\alpha,\beta)$-parameter space, while the fourth case ($l_1 \cap r$) leads to an empty intersection, since $l_{\xi^\ast}$ and $r_{\xi^\ast}$ are separated from $\mathcal{L}_{\xi^\ast}^+$ and $\mathcal{R}_{\xi^\ast}^-$, respectively, only for $\alpha \leq \rho_c^0$ and $\beta \leq 1-\rho_c^1$, and in this case they can never coincide.
Thus, we are left with:
\begin{description}
 \item[Case 1: $l_{\xi^\ast} \cap \mathcal{R}_{\xi^\ast}^- \neq \emptyset$.] From the investigation of this case we obtain orbits of type 1, 2.
 \item[Case 2: $\mathcal{L}_{\xi^\ast}^+ \cap \mathcal{R}_{\xi^\ast}^- \neq \emptyset$.] From the investigation of this case we obtain orbits of type 3, 4, 5, 6.
 \item[Case 3: $\mathcal{L}_{\xi^\ast}^+ \cap r_{\xi^\ast} \neq \emptyset$.] From the investigation of this case we obtain orbits of type 7, 8.
\end{description}
% We point out that in our arguments below we use the fact that $\frac{k(0)}{k(1)}>1$, deriving from assumption \eqref{eq:g}. 
In the following, we examine Cases 1-3 in more detail.\\

\emph{Case 1: $l_{\xi^\ast}\,\cap\,\mathcal{R}_{\xi^\ast}^- \neq \emptyset$.} By definition of $l_{\xi^\ast}$, this occurs only when $\alpha \leq \rho_c^0$. In this case, we have $l_{\xi^\ast} \in \mathcal{R}_{\xi^\ast}^-$, which implies that $p_0 = l$ and, consequently, $\rho_0=\alpha$. This implies that in this regime no boundary layers exist at $\xi=0$. Moreover, since $\rho(1,s)=\rho^\ast(\alpha)$, following the flow of the layer problem until it hits $\mathcal{R}$ we obtain %$\rho^+(\xi^\ast,\alpha)=\rho^-(\xi^\ast,\rho^\ast(\alpha))$
\begin{equation} \label{eq:rho1_case1}
 \rho_1 = \frac{\alpha(1-\alpha)k(0)}{\beta k(1)}.
\end{equation}
In this case, the singular orbit consists in a slow motion along $\mathcal{C}_0^r$ followed by a layer at $\xi=1$. The nature of this layer -- in particular its orientation -- depends on $\alpha$ and $\beta$ as follows:
\begin{itemize}
 \item When $\alpha < \rho_c^0$ and $\rho^\ast(\alpha) < \beta < 1-\rho^\ast(\alpha)$, i.e.~for $(\alpha,\beta) \in \mathcal{G}_1$, $\rho$ increases along the boundary layer at $\xi=1$. The corresponding singular solution is therefore of type $1$ (see Figure \ref{fig:singsol1}(a)).
 \item When $\alpha < \rho_c^0$ and $\beta > 1-\rho^\ast(\alpha)$, i.e.~for $(\alpha,\beta) \in \mathcal{G}_2$, $\rho$ decreases along the boundary layer at $\xi=1$. Therefore, the corresponding singular solution is of type $2$ (see Figure \ref{fig:singsol1}(b)). 
\end{itemize}
We note that when $\alpha < \rho_c^0$ and $\beta=1-\rho^\ast(\alpha)$ (i.e.~on $\gamma_{12}$) there is no layer at $\xi=1$.\\

\emph{Case 2: $\mathcal{L}_{\xi^\ast}^+ \cap \mathcal{R}_{\xi^\ast}^- \neq \emptyset$.} We observe that by definition $\mathcal{L}_{\xi^\ast}^+ \subset \mathcal{C}_0^a$ and $\mathcal{R}_{\xi^\ast}^- \subset \mathcal{C}_0^r$. Thus, this case corresponds to having $\alpha \geq \rho_c^0$ and $\beta \geq 1-\rho_c^1$ and their non-empty intersection is realised at the canard point $p^\ast$ (see \eqref{eq:P}). This implies that the slow segment of these singular orbits is the canard orbit $S_c$.\\
In particular, since $\rho^+(\xi^\ast,s)=\frac12=\rho^-(\xi^\ast,t)$, it follows that $\rho^+(0,s)=1-\rho_c^0$ and $\rho^-(1,t)=1-\rho_c^1$. Consequently, the start/end point of the reduced flow are fixed by the canard and correspond to $p_c^0$ and $p_c^1$ respectively, whereas boundary layers at $\xi=0$, $1$ may arise depending on $\alpha$ and $\beta$. It is then possible to determine the starting and ending points of the orbit by following the flow of the layer problem (backwards at $\xi=0$ and forward at $\xi=1$); this leads to
\begin{equation} \label{eq:rho01_case2}
 \rho_0 = 1-\frac{k(\xi^\ast)}{4 \alpha k(0)}, \qquad \rho_1 = \frac{k(\xi^\ast)}{4 \beta k(1)}.
\end{equation}
In particular, we have:
\begin{itemize}
 \item When $\rho_c^0 < \alpha < 1-\rho_c^0$, $\rho$ increases along the boundary layer at $\xi=0$. Additionally:
 \begin{itemize}
 \item If $1-\rho_c^1 < \beta < \rho_c^1$, i.e.~for $(\alpha,\beta) \in \mathcal{G}_3$, $\rho$ increases along the boundary layer at $\xi=1$. This implies that the singular orbit is of type $3$ (see Figure \ref{fig:singsol1}(c)).
 \item If $\rho_c^1 < \beta < 1$, i.e.~for $(\alpha,\beta) \in \mathcal{G}_4$, $\rho$ decreases along the boundary layer at $\xi=1$. This implies that the singular orbit is of type $4$ (see Figure \ref{fig:singsol1}(d)).
 \end{itemize}
 \item When $1-\rho_c^0 < \alpha < 1$, $\rho$ decreases along the boundary layer at $\xi=0$. Additionally:
 \begin{itemize}
 \item If $1-\rho_c^1 < \beta < \rho_c^1$, i.e.~for $(\alpha,\beta) \in \mathcal{G}_5$, $\rho$ increases along the boundary layer at $\xi=1$. This implies that the singular orbit is of type $5$ (see Figure \ref{fig:singsol2}(a)).
 \item If $\rho_c^1 < \beta < 1$, i.e.~for $(\alpha,\beta) \in \mathcal{G}_6$, $\rho$ decreases along the boundary layer at $\xi=1$. This implies that the singular orbit is of type $6$ (see Figure \ref{fig:singsol2}(b)).
 \end{itemize}
\end{itemize}
We note that when $\alpha=1-\rho_c^0$ and $\beta \geq 1-\rho_c^1$ (i.e.~on $\gamma_{35} \cup \gamma_{46}$) we have no boundary layer at $\xi=0$. Moreover, when $\beta=\rho_c^1$ and $\alpha \geq \rho_c^0$ (i.e.~on $\gamma_{34} \cup \gamma_{56}$) we have no boundary layer at $\xi=1$.\\

\emph{Case 3: $\mathcal{L}_{\xi^\ast}^+ \cap r_{\xi^\ast} \neq \emptyset$.} By definition of $r_{\xi^\ast}$, this occurs only when $\beta < 1-\rho_c^1$. In this case, we have $r_{\xi^\ast} \in \mathcal{L}_{\xi^\ast}^+$, which implies that $p_1 = r$ and, consequently, $\rho_1=1-\beta$ (i.e., no boundary layers emerge at $\xi=1$). Moreover, since $\rho(0,s)=\rho_\ast(\beta)$, following the layer problem backwards until it hits $\mathcal{L}$, we obtain
\begin{equation} \label{eq:rho0_case3}
 \rho_0=1-\frac{\beta(1-\beta)k(1)}{\alpha k(0)}.
\end{equation}
Consequently, the slow motion is here entirely contained in $\mathcal{C}_0^a$ and there is a boundary layer at $\xi=0$, whose nature depends on $\alpha$ as follows:
\begin{itemize}
 \item If $1-\rho_\ast(\beta) < \alpha < \rho_\ast(\beta)$ and $0 < \beta < 1-\rho_c^1 $, i.e.~if $(\alpha,\beta) \in \mathcal{G}_7$, $\rho$ is increasing and the singular solution is of type $7$ (see Figure \ref{fig:singsol2}(c)).
 \item If $\rho_\ast(\beta) < \alpha < 1$ and $0 < \beta < 1-\rho_c^1$, i.e.~if $(\alpha,\beta) \in \mathcal{G}_8$, $\rho$ is decreasing, and we have a singular solution of type $8$ (see Figure \ref{fig:singsol2}(d)).
\end{itemize}
We note that when $\alpha=\rho_\ast(\beta)$ and $\beta < 1-\rho_c^1$ (i.e.~on $\gamma_{78}$), there are no boundary layers.
% We also remark that when $\alpha > \rho_c^0$ and $\beta = 1-\rho_c^1$ (i.e.~on $\gamma_{37} \cup \gamma_{58}$), we have a unique solution with no boundary layer at $\xi=1$ satisfying
%  \begin{equation}
%   \rho(0,s)=1-\rho_c, \quad \rho(1,s)=\frac12.
%  \end{equation}
\end{proof}

\begin{remark}
 The construction in Case 3 is essentially the same as the one in Case 1 upon reversal of the flow direction in \eqref{eq:sys}.
\end{remark}

\begin{remark} \label{rem:oldpap}
 Singular solutions of type 1, 2, 7, and 8 can be obtained also applying the same strategy used in \cite[Proposition 2]{Iuorio_2022}, as their slow portion is entirely contained in one of the two halves of the critical manifold ($\mathcal{C}_0^r$ in the case of type 1, 2, $\mathcal{C}_0^r$ in the case of type 7, 8). Therefore, it would be possible to only focus on the flow of the manifold $\mathcal{L}$ of left boundary conditions up to $\xi=1$ and check its intersection with the projection of the manifold $\mathcal{R}$ of right boundary conditions on $\mathcal{C}_0$.
\end{remark}

\begin{remark} \label{rem:keffect}
 Different values of $k(0)$, $k(1)$, and $k(\xi^\ast)$ influence the structure of the bifurcation diagram sketched in \ref{fig:bifdiag} only quantitatively. In particular, the smaller $k(\xi^\ast)$ is, the larger regions $\mathcal{G}_i$, $i=3,4,5,6$ are, consequently reducing the sizes of regions $\mathcal{G}_i$, $i=1,2,7,8$. Recall that smaller values of $k(\xi^\ast)$ correspond to a narrower bottleneck.
\end{remark}

% \begin{tcolorbox}[colback=red!5,colframe=red!75!black]
%  Considerations:
%  \begin{itemize}
%   \item Add computations and plots for different regions \& singular solutions
%   \item Compare with numerical results
%  \end{itemize}
% \end{tcolorbox}

\begin{figure}[!ht]
    \centering
    \includegraphics[scale=0.5]{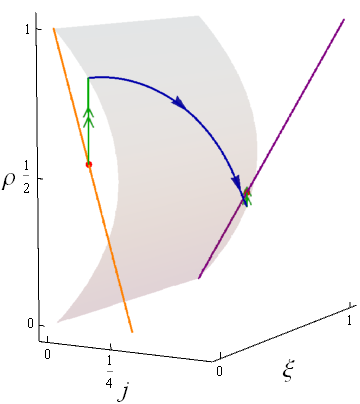}
    \caption{Schematic representation of a singular solution of \eqref{eq:sys}-\eqref{eq:bc} with $a=0.3$, $b=1.5$, $\alpha=0.3$, and $\beta=0.6$ (i.e. $(\alpha,\beta) \in \mathcal{G}_3$). The solution has boundary layers at $\xi=0$ and $\xi=1$, while the slow portion of the orbit coincides with $S_c$.}
    \label{fig:ss_canard}
\end{figure}

\begin{figure}
\centering
 (a) Region $\mathcal{G}_1$: $\alpha=0.1$, $\beta=0.4$\\
 \includegraphics[scale=0.45]{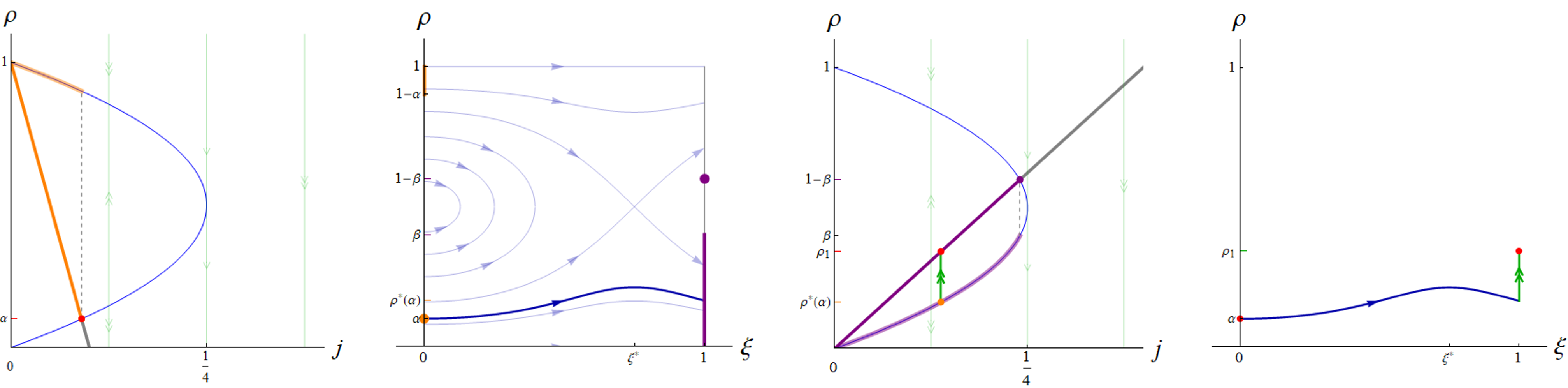}\\
 %\includegraphics[scale=0.22]{plots_knotsymm/T1_leftBC_alpha0.1_beta0.4.png}
 %\hspace{.5cm}
 %\includegraphics[scale=0.22]{plots_knotsymm/T1_alpha0.1_beta0.4_arr.png}
 %\hspace{.5cm}
 %\includegraphics[scale=0.22]{plots_knotsymm/T1_rightBC_alpha0.1_beta0.4_arr.png}
 %\hspace{.5cm}
 %\includegraphics[scale=0.22]{plots_knotsymm/T1_singsol_alpha0.1_beta0.4_arr.png}\\
 \vspace{.5cm}
 (b) Region $\mathcal{G}_2$: $\alpha=0.1$, $\beta=0.9$\\
 \includegraphics[scale=0.45]{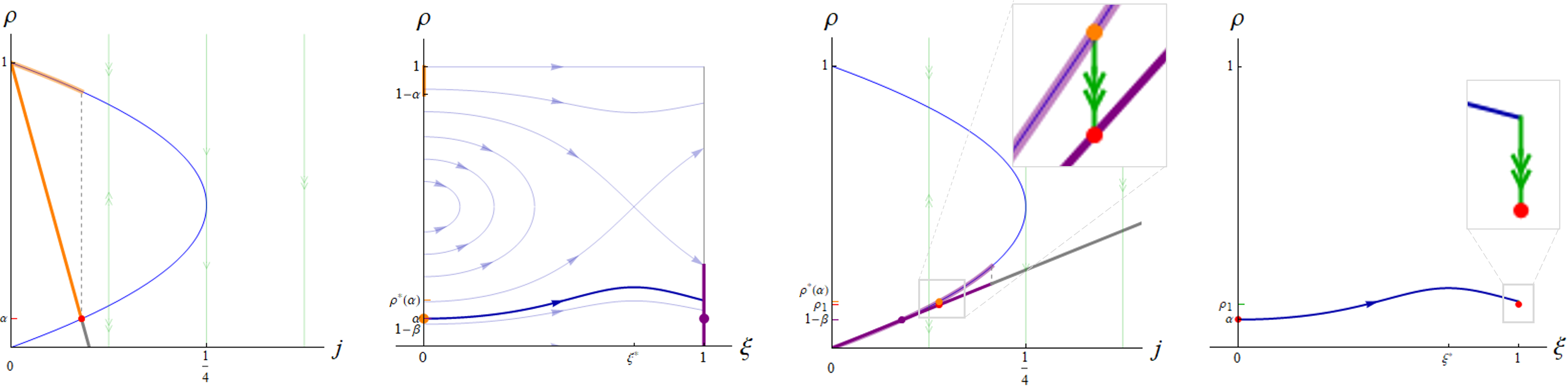}\\
 %\includegraphics[scale=0.22]{plots_knotsymm/T2_leftBC_alpha0.1_beta0.9.png}
 %\hspace{.5cm}
 %\includegraphics[scale=0.22]{plots_knotsymm/T2_alpha0.1_beta0.9_arr.png}
 %\hspace{.5cm}
 %\includegraphics[scale=0.22]{plots_knotsymm/T2_rightBC_alpha0.1_beta0.9.png}
 %\hspace{.5cm}
 %\includegraphics[scale=0.22]{plots_knotsymm/T2_singsol_alpha0.1_beta0.9.png}\\
 \vspace{.5cm}
 (c) Region $\mathcal{G}_3$: $\alpha=0.3$, $\beta=0.6$\\
 \includegraphics[scale=0.45]{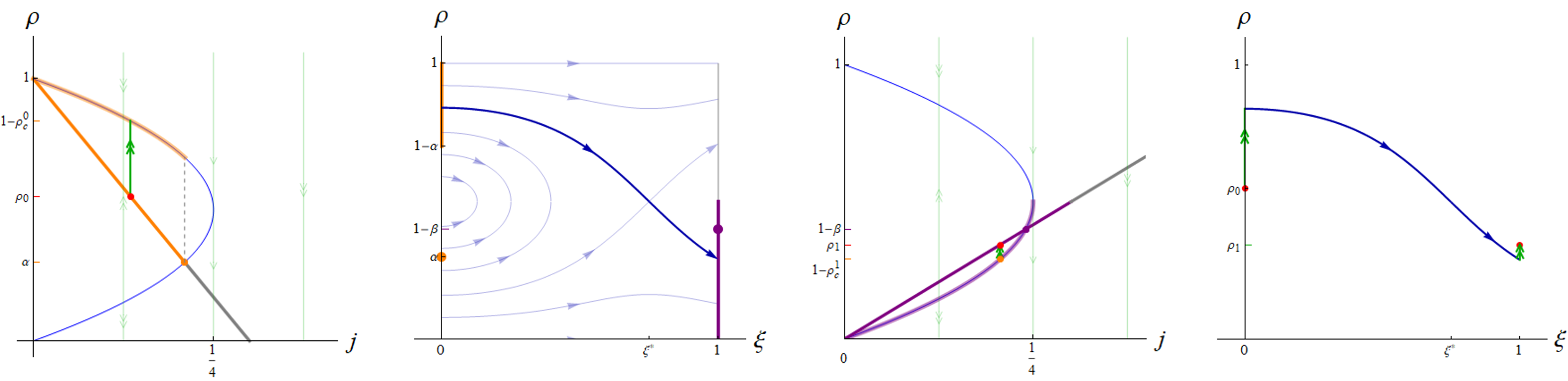}\\
 %\includegraphics[scale=0.22]{plots_knotsymm/T3_leftBC_alpha0.3_beta0.6_arr.png}
 %\hspace{.5cm}
 %\includegraphics[scale=0.22]{plots_knotsymm/T3_alpha0.3_beta0.6_arr.png}
 %\hspace{.5cm}
 %\includegraphics[scale=0.22]{plots_knotsymm/T3_rightBC_alpha0.3_beta0.6_arr.png}
 %\hspace{.5cm}
 %\includegraphics[scale=0.22]{plots_knotsymm/T3_singsol_alpha0.3_beta0.6_arr.png}\\
 \vspace{.5cm}
 (d) Region $\mathcal{G}_4$: $\alpha=0.3$, $\beta=0.8$\\
 \includegraphics[scale=0.45]{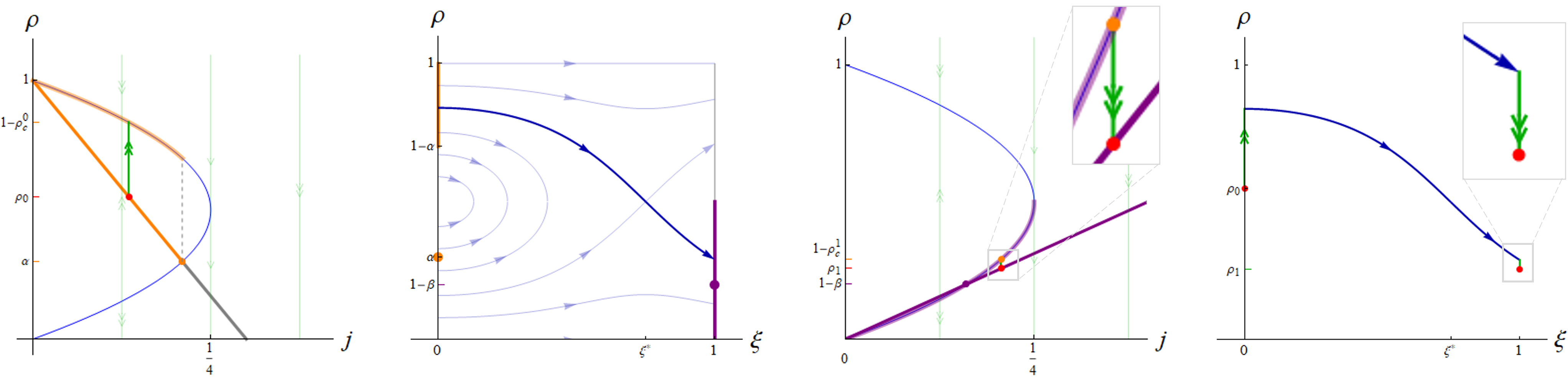}\\
 %\includegraphics[scale=0.22]{plots_knotsymm/T4_leftBC_alpha0.3_beta0.8_arr.png}
 %\hspace{.5cm}
 %\includegraphics[scale=0.22]{plots_knotsymm/T4_alpha0.3_beta0.8_arr.png}
 %\hspace{.5cm}
 %\includegraphics[scale=0.22]{plots_knotsymm/T4_rightBC_alpha0.3_beta0.8_arr.png}
 %\hspace{.5cm}
 %\includegraphics[scale=0.22]{plots_knotsymm/T4_singsol_alpha0.3_beta0.8_arr.png}
 \caption{Schematic representation of singular solutions of type 1-4 (rows 1-4, respectively). \emph{First column}: Boundary conditions at $\xi=0$ in $(j,\rho)$-space: the orange line is $\mathcal{L}$, while the orange curve is $\mathcal{L}^+$. The red dot represents $p_0$ and the green line illustrates the layer where $\rho$ increases (type 3, 4). \emph{Second column}: Slow evolution on $\mathcal{C}_0$ (blue curve). The orange lines are the projection of $\mathcal{L}$ and $\mathcal{L}^+$ on $\mathcal{C}_0$, while the purple one represents the projection of $\mathcal{R}^-$ on $\mathcal{C}_0$. The orange dot corresponds to $l$, while the purple dot corresponds to $r$. For orbits of type 3 and 4 the slow flow involves the passage through the canard point $p^\ast$. \emph{Third column}: Boundary conditions at $\xi=1$ in $(j,\rho)$-space. The red dot corresponds to $p_1$, while the purple line and curve represent the manifolds $\mathcal{R}$ and $\mathcal{R}^-$, respectively. The green line corresponds to the layer of the singular orbit where $\rho$ increases (type 1-3)/decreases (type 2-4). \emph{Fourth column}: Singular solution in $(\xi,\rho)$-space.}
 \label{fig:singsol1}
\end{figure}

\begin{figure}
\centering
 (a) Region $\mathcal{G}_5$: $\alpha=0.9$, $\beta=0.6$\\
 \includegraphics[scale=0.45]{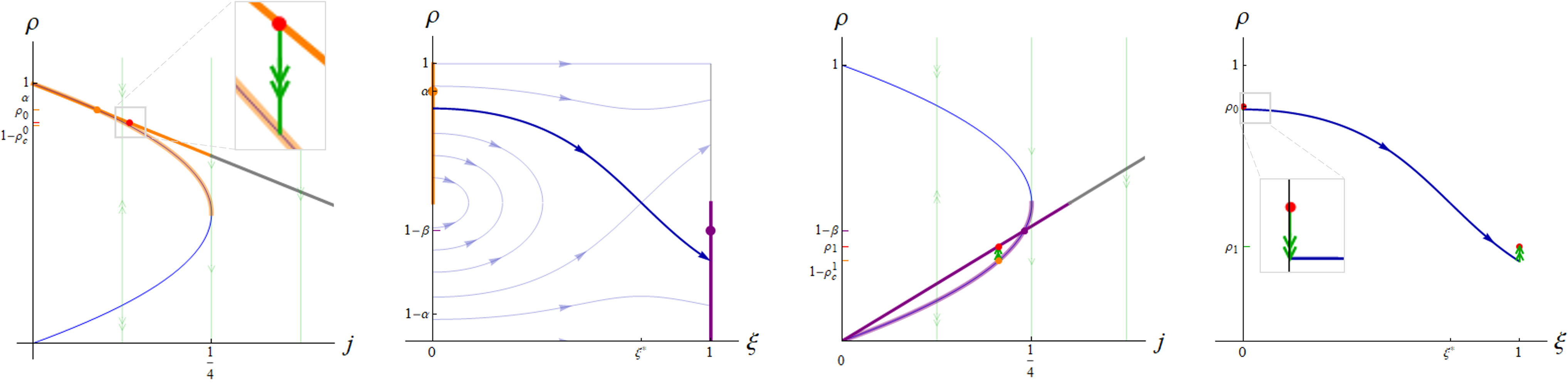}\\
 %\includegraphics[scale=0.22]{plots_knotsymm/T5_leftBC_alpha0.9_beta0.6.png}
 %\hspace{.5cm}
 %\includegraphics[scale=0.22]{plots_knotsymm/T5_alpha0.9_beta0.6_arr.png}
 %\hspace{.5cm}
 %\includegraphics[scale=0.22]{plots_knotsymm/T5_rightBC_alpha0.9_beta0.6_arr.png}
 %\hspace{.5cm}
 %\includegraphics[scale=0.22]{plots_knotsymm/T5_singsol_alpha0.9_beta0.6_arr.png}\\
 \vspace{.5cm}
 (b) Region $\mathcal{G}_6$: $\alpha=0.9$, $\beta=0.8$\\
 \includegraphics[scale=0.45]{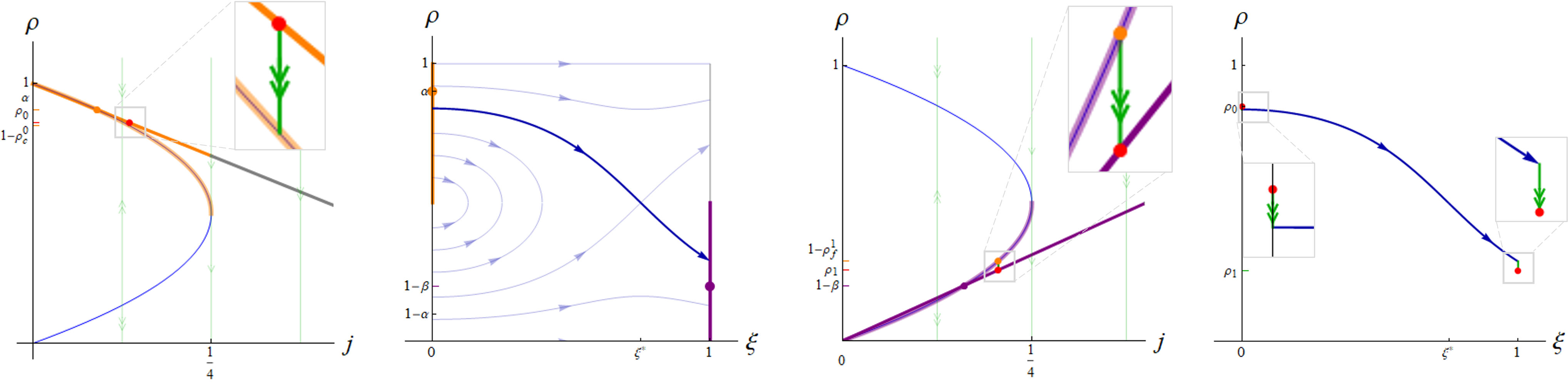}\\
 %\includegraphics[scale=0.22]{plots_knotsymm/T6_leftBC_alpha0.9_beta0.8.png}
 %\hspace{.5cm}
 %\includegraphics[scale=0.22]{plots_knotsymm/T6_alpha0.9_beta0.8_arr.png}
 %\hspace{.5cm}
 %\includegraphics[scale=0.22]{plots_knotsymm/T6_rightBC_alpha0.9_beta0.8_arr.png}
 %\hspace{.5cm}
 %\includegraphics[scale=0.22]{plots_knotsymm/T6_singsol_alpha0.9_beta0.8_arr.png}\\
 \vspace{.5cm}
 (c) Region $\mathcal{G}_7$: $\alpha=0.7$, $\beta=0.2$\\
 \includegraphics[scale=0.45]{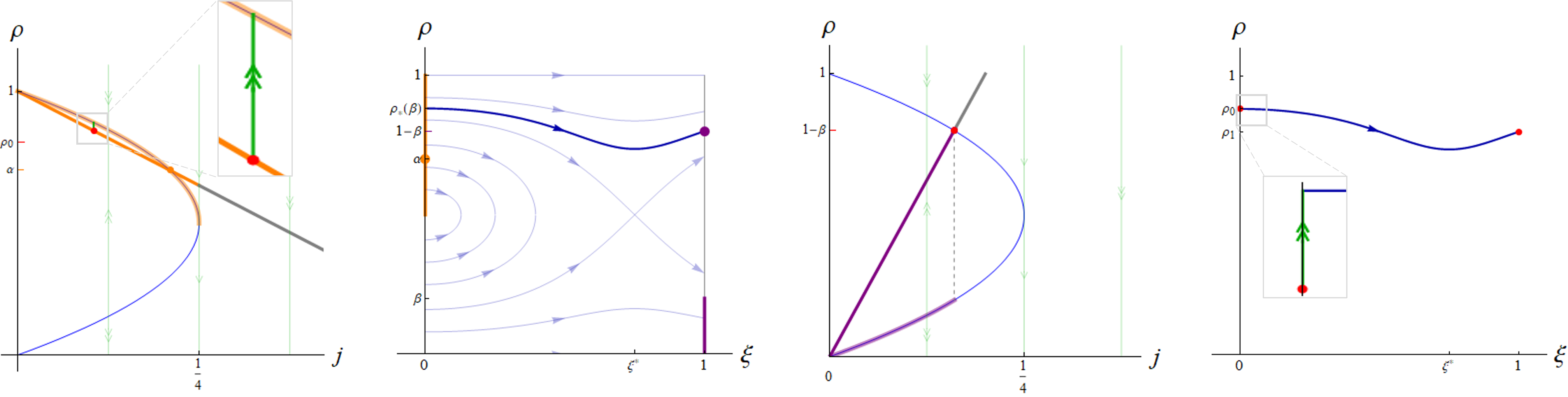}\\
 %\includegraphics[scale=0.22]{plots_knotsymm/T7_leftBC_alpha0.7_beta0.2.png}
 %\hspace{.5cm}
 %\includegraphics[scale=0.22]{plots_knotsymm/T7_alpha0.7_beta0.2_arr.png}
 %\hspace{.5cm}
 %\includegraphics[scale=0.22]{plots_knotsymm/T7_rightBC_alpha0.7_beta0.2.png}
 %\hspace{.5cm}
 %\includegraphics[scale=0.22]{plots_knotsymm/T7_singsol_alpha0.7_beta0.2_arr.png}\\
 \vspace{.5cm}
 (d) Region $\mathcal{G}_8$: $\alpha=0.9$, $\beta=0.2$\\
 \includegraphics[scale=0.45]{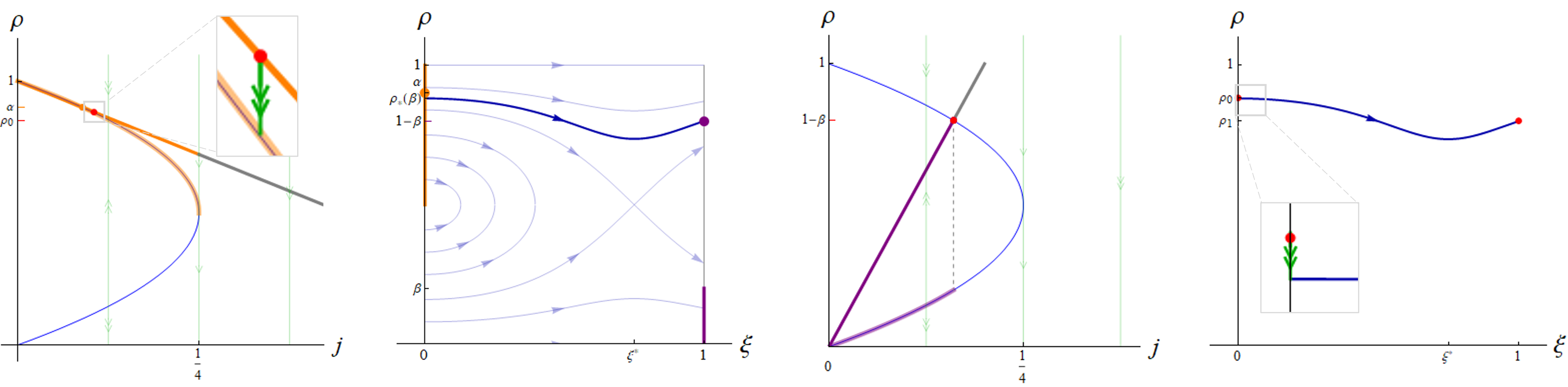}
 %\includegraphics[scale=0.22]{plots_knot%symm/T8_leftBC_alpha0.9_beta0.2.png}
 %\hspace{.5cm}
 %\includegraphics[scale=0.22]{plots_knotsymm/T8_alpha0.9_beta0.2_arr.png}
 %\hspace{.5cm}
 %\includegraphics[scale=0.22]{plots_knotsymm/T8_rightBC_alpha0.9_beta0.2.png}
 %\hspace{.5cm}
 %\includegraphics[scale=0.22]{plots_knotsymm/T8_singsol_alpha0.9_beta0.2_arr.png}
 \caption{Schematic representation of singular solutions of type 5-8 (rows 1-4, respectively). \emph{First column}: Boundary conditions at $\xi=0$ in $(j,\rho)$-space: the orange line is $\mathcal{L}$, while the orange curve is $\mathcal{L}^+$. The red dot represents $p_0$ and the green line illustrates the layer where $\rho$ increases (type 7)/decreases (type 5, 6, 8). \emph{Second column}: Slow evolution on $\mathcal{C}_0$ (blue curve). The orange lines are the projection of $\mathcal{L}$ and $\mathcal{L}^+$ on $\mathcal{C}_0$, while the purple one represents the projection of $\mathcal{R}^-$ on $\mathcal{C}_0$. The orange dot corresponds to $l$, while the purple dot corresponds to $r$. For orbits of type 5 and 6 the slow flow involves the passage through the canard point $p^\ast$. \emph{Third column}: Boundary conditions at $\xi=1$ in $(j,\rho)$-space. The red dot corresponds to $p_1$, while the purple line and curve represent the manifolds $\mathcal{R}$ and $\mathcal{R}^-$, respectively. The green line corresponds to the layer of the singular orbit where $\rho$ increases (type 5)/decreases (type 6). \emph{Fourth column}: Singular solution in $(\xi,\rho)$-space.}
 \label{fig:singsol2}
\end{figure}

\begin{remark}[Degenerate cases including continua of singular solutions] \label{rem:fmam}
When $\alpha \leq \rho_c^0$ and $\beta=\rho^\ast(\alpha)$ - i.e.~when $(\alpha,\beta) \in \gamma_{17}$ - we have that both $\mathcal{L}_{\xi^\ast}^+ \cap r_{\xi^\ast}$ and $l_{\xi^\ast} \cap \mathcal{R}_{\xi^\ast}^-$ are non-empty. Consequently, there are two possible reduced solutions, satisfying (see Figure \ref{fig:sp_cases}(a))
  \begin{equation}
  \text{(a)} \, \begin{cases}
              \rho(0,s)=\alpha,\\
              \rho(1,s)=\beta,
             \end{cases}
  \text{ or } \quad \text{ (b) } \begin{cases}
              \rho(0,s)=1-\alpha,\\
              \rho(1,s)=\rho_1=1-\beta.
             \end{cases}
 \end{equation}
 In this case, we have a continuum of singular solutions, since at any $\xi \in [0,1]$ it is possible to jump from the slow trajectory of the reduced flow in (a) to the one in (b) via the flow of the layer problem. Analogously, we obtain a continuum of singular solutions when $\alpha=\rho_c^0$, $\beta \geq 1-\rho_c^1$, i.e.~when $(\alpha,\beta) \in \gamma_{13} \cup \gamma_{24}$. In this case, in fact, we have that both $\mathcal{L}_{\xi^\ast}^+ \cap \mathcal{R}_{\xi^\ast}^-$ and $l_{\xi^\ast} \cap \mathcal{R}_{\xi^\ast}^-$ are non-empty, and therefore there are two possible reduced solutions (with jumps possible at any $\xi \in [0,\xi^\ast]$ via the flow of the layer problem) satisfying (see Figure \ref{fig:sp_cases}(b))
 \begin{equation}
  \text{(c)} \, \begin{cases}
              \rho(0,s)=\rho_c,\\
              \rho(1,s)=1-\rho_c^1,
             \end{cases}
  \text{ or } \quad \text{ (d) } \begin{cases}
              \rho(0,s)=1-\rho_c,\\
              \rho(1,s)=1-\rho_c^1.
             \end{cases}
 \end{equation}
 A last example of such a situation is given by $\alpha \geq \rho_c^0$, $\beta = 1-\rho_c^1$, i.e.~when $(\alpha,\beta) \in \gamma_{37} \cup \gamma_{58}$. Here,  both $\mathcal{L}_{\xi^\ast}^+ \cap \mathcal{R}_{\xi^\ast}^-$ and $\mathcal{L}_{\xi^\ast}^+ \cap r_{\xi^\ast}$ are non-empty, leading again to two possible reduced solutions (with jumps possible at any $\xi \in [\xi^\ast,1]$ via the flow of the layer problem) satisfying (see Figure \ref{fig:sp_cases}(c))
 \begin{equation}
  \text{(c)} \, \begin{cases}
              \rho(0,s)=1-\frac{k(\xi^\ast)}{4 \alpha k(0)},\\
              \rho(1,s)=\rho_c^1,
             \end{cases}
  \text{ or } \quad \text{ (d) } \begin{cases}
              \rho(0,s)=1-\frac{k(\xi^\ast)}{4 \alpha k(0)},\\
              \rho(1,s)=1-\rho_c^1.
             \end{cases}
 \end{equation}
 Since in these degenerate cases singular solutions are not unique, our method based on transversality arguments to infer persistence of singular solutions to \eqref{eq:reduced equation rho}-\eqref{eq:boundary conditions reduced rho} for $0 < \varepsilon \ll 1$ do not apply. Moreover, at the point $\alpha = \rho_c^0$, $\beta=1-\rho_c^1$ -- i.e. at the intersection of $\gamma_{17}$, $\gamma_{13}$, and $\gamma_{37}$ -- the situation is even more degenerate as the three previous scenarios collide. We leave the analysis of these more delicate situations for future work.
\end{remark}

\begin{figure}[!ht]
      \centering
     \begin{minipage}{0.3\textwidth}
          \centering
          \includegraphics[scale=.35]{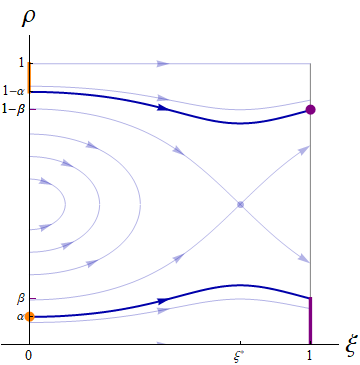}
          
          (a)
     \end{minipage}
     \hfill
%      \raggedleft
     \begin{minipage}{0.3\textwidth}
          \centering
          \includegraphics[scale=.35]{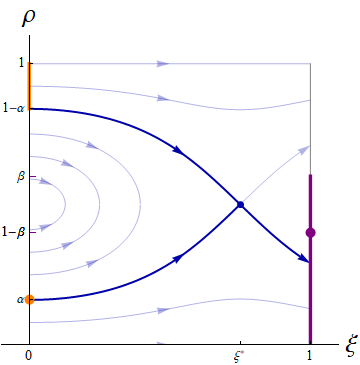}
          
          (b)
     \end{minipage}
          \hfill
%      \raggedleft
     \begin{minipage}{0.3\textwidth}
          \centering
          \includegraphics[scale=.35]{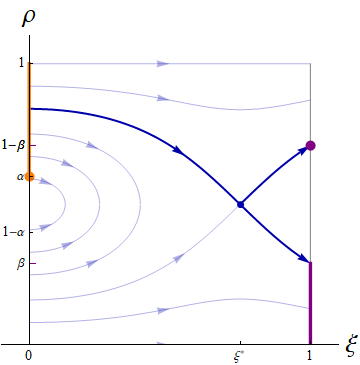}
          
          (c)
     \end{minipage}
     \caption{Schematic representation in $(\xi,\rho)$-space of the slow portions (blue curves) of the possible singular orbits for (a) $(\alpha,\beta) \in \gamma_{17}$, (b) $(\alpha,\beta) \in \gamma_{13} \cup \gamma_{24}$, and (c) $(\alpha,\beta) \in \gamma_{37} \cup \gamma_{58}$. The orange and purple curves correspond to the projection of $\mathcal{L}^+$ and $\mathcal{R}^-$, respectively, on the $(\xi,\rho)$-space. Fast jumps from the slow solution in $\mathcal{C}_0^r$ to the slow solution in $\mathcal{C}_0^a$ are possible (a) at each $\xi \in [0,1]$, (b) for $\xi \in [0,\xi^\ast]$, and (c) for $\xi \in [\xi^\ast,1]$.}
     \label{fig:sp_cases}
\end{figure}

We now prove that the singular solutions from Proposition \ref{prop:singsol} perturb to solutions of \eqref{eq:reduced equation rho}-\eqref{eq:boundary conditions reduced rho} for $\varepsilon$ sufficiently small.

\begin{theorem} \label{thm:k}
 Let $k \in C^2([0,1])$ be a positive function satisfying Assumption \eqref{k-assumption}.
 For each $(\alpha,\beta) \in \mathcal{G}_i$, $i=1,\dots,8$, the boundary value problem \eqref{eq:reduced equation} has a unique solution $\rho(x,\alpha,\beta,\varepsilon)$ for $\varepsilon$ sufficiently small. In the phase-space formulation \eqref{eq:sys_fast}, this solution corresponds to an orbit $\Gamma^i_\varepsilon$ which is $\mathcal{O}(\varepsilon^\mu)$-close to $\Gamma^i$ in terms of Hausdorff distance, with $\mu=1$ for $i=1,2,7,8$ and $\mu=1/2$ for $i=3,4,5,6$.
\end{theorem}
\begin{proof}
The solutions for $\varepsilon$ small are obtained by perturbing from the singular solutions $\Gamma^i$, $i=1,\dots,8$. More precisely, we show that the manifold obtained by flowing the line $\mathcal{L}$ of points corresponding to the boundary conditions at $\xi=0$ to $\xi=\xi^\ast$ for $\varepsilon$ small intersects the manifold obtained by flowing the line $\mathcal{R}$ of points corresponding to the boundary conditions at $\xi=1$ to $\xi=\xi^\ast$ in a point which is close to the corresponding point of the singular solution. Analogously to Proposition \ref{prop:singsol}, this is done by considering three cases.\\

\emph{Case 1: $(\alpha,\beta) \in \mathcal{G}_i$, $i=1,2$}. In this case,
the proof is completely analogous to Case 1 in \cite[Theorem 2]{Iuorio_2022}. In particular, it is possible to show that for $0 < \varepsilon \ll 1$ the (forward) flow defined by \eqref{eq:sys} takes a suitable small segment of $\mathcal{L}$ to a smooth, two-dimensional manifold $\mathcal{M}_{0,\varepsilon}$, which reduces to a curve $\mathcal{L}_{1,\varepsilon}$ when projected in the plane $\xi=1$. Such curve intersects $\mathcal{R}$ in a point $p_{1,\varepsilon}$ which corresponds to the right end-point of the solution of the boundary value problem. The full solution for $\xi \in [0,1]$ is then obtained by following the flow backward from $p_{1,\varepsilon}$ to $\xi=0$. In this case, the perturbed orbits are $\mathcal{O}(\varepsilon)$ close to the corresponding singular ones as all perturbations are $C^1$ in $\varepsilon$.\\

\emph{Case 2: $(\alpha,\beta) \in \mathcal{G}_i$, $i=3,4,5,6$}. In this case, the singular solution starts with a layer connecting the point $p_0 \in \mathcal{L}$ to the point $p_c^0$ on $S_c$, then follows the canard through the canard point $p^\ast$ up to $\xi=1$, and finally ends with another layer connecting $p_c^1$ with the point $p_1 \in \mathcal{R}$.\\
To prove the persistence of this singular orbit, we flow the line $\mathcal{L}$ of boundary conditions at $\xi=0$ forward, the line $\mathcal{R}$ of boundary conditions at $\xi=1$ backward, and show that they intersect transversally at $\xi = \xi^\ast$ for $\varepsilon$ small. Since the singular solution involves the point $p^\ast$ on the non-hyperbolic fold line $F$ and the emergence of a canard, results on extending GSPT to such problems \cite{Szmolyan_2001} are needed here. 

Fenichel theory \cite{Fenichel_1979} implies that away from the fold line $F$ (compact subsets of) $\mathcal{C}_0^a$ and $\mathcal{C}_0^r$ perturb smoothly to the slow manifolds $\mathcal{C}_\varepsilon^a$ and $\mathcal{C}_\varepsilon^r$, respectively. The results in \cite[Theorem 4.1]{Szmolyan_2001} imply that in a neighbourhood of the canard point $p^\ast$ the manifolds $\mathcal{C}_\varepsilon^a$ and $\mathcal{C}_\varepsilon^r$ intersect transversally in a maximal canard $S_c^\varepsilon$ (close to $S_c$) for $\varepsilon$ sufficiently small. As in case 1, consider a small segment of $\mathcal{L}$ containing $p_0$ and denote its extension by the forward flow of \eqref{eq:sys_fast} by $\mathcal{M}_{0,\varepsilon}$ for $\varepsilon$ small. Analogously, consider a small segment of $\mathcal{R}$ containing $p_1$ and denote its extension by the backward flow of \eqref{eq:sys_fast} by $\mathcal{M}_{1,\varepsilon}$ for $\varepsilon$ small (again a smooth, two-dimensional manifold). By Fenichel theory, the manifolds $\mathcal{M}_{0,\varepsilon}$ and $\mathcal{M}_{1,\varepsilon}$ are exponentially close to $\mathcal{C}_\varepsilon^a$ and $\mathcal{C}_\varepsilon^r$, respectively. Therefore, $\mathcal{M}_{0,\varepsilon}$ and $\mathcal{M}_{1,\varepsilon}$ also intersect transversally in a unique orbit, which is the unique solution to the boundary value problem (see Figure \ref{fig:thm_case2}). Here, the
$\mathcal{O}(\varepsilon^{1/2})$ distance between the perturbed and the corresponding singular solutions follows from the blow-up analysis in \cite{Szmolyan_2001}, since the effect of the perturbation in the scaling chart of the blow-up transformation is  of the order $\varepsilon^{1/2}$. \\

\begin{figure}[!ht]
	\centering
	\begin{overpic}[scale=.7]{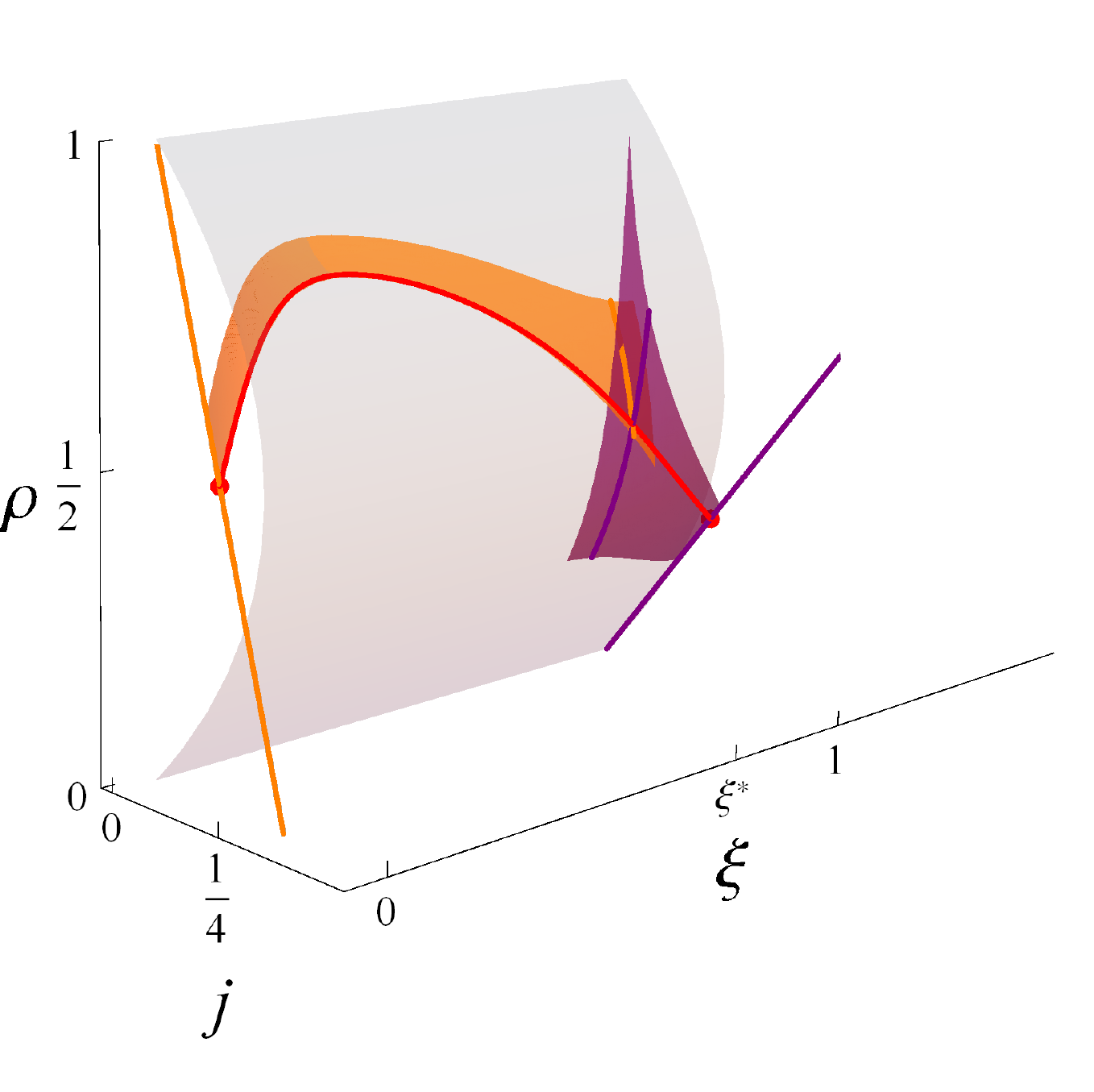}
% 	\put(71,55){\small{$p_{1,\varepsilon}$}}
% 	\put(66,49){\small{$p_1$}}
    \put(30,80){$\mathcal{M}_{0,\varepsilon}$}
    \put(60,70){$\mathcal{M}_{1,\varepsilon}$}
	\end{overpic}
\caption{Schematic representation of a solution (continuous red curve) to the full problem \eqref{eq:reduced equation rho}-\eqref{eq:boundary conditions reduced rho} obtained for $(\alpha,\beta) \in \mathcal{G}_4$ with the strategy discussed in Case 2 of Theorem \ref{thm:k}. The orange and purple manifolds represent $\mathcal{M}_{0,\varepsilon}$ in $[0,\xi^\ast+\eta]$ and $\mathcal{M}_{1,\varepsilon}$ in  $[\xi^\ast-\eta,1]$ with $\eta=1/18$, respectively, whereas the red dots at $\xi=0$ and $\xi=1$ correspond to the initial and final point of the orbit. These manifolds intersect transversally at $\xi = \xi^\ast$, providing the uniqueness of the obtained solution. The passage close to the canard point $p^\ast$ leads to the $\mathcal{O}(\varepsilon^{1/2})$ distance from the corresponding singular orbit.}
\label{fig:thm_case2}
\end{figure}

\emph{Case 3: $(\alpha,\beta) \in \mathcal{G}_i$, $i=7,8$}. This case can be proved following the same approach as in \cite[Theorem 2]{Iuorio_2022}, and in particular is completely analogous to Case 1 upon reversal of the flow direction.
\end{proof}

\section{Numerical experiments}\label{sec:numerics}

In this section we present some numerical results for the steady-state problem \eqref{eq:reduced equation} which support the analysis of Section~\ref{sec:gspt}. More details about the numerical method employed here can be found in~\cite{Iuorio_2022}. All results are obtained for $\varepsilon = 10^{-3}$.\\

We first set $k(\xi) = 1 + a \cos\left(\frac{2\pi\xi}{b}\right)$, a choice that was already considered in Figure~\ref{fig:bifdiag} in the singular case. 
For $\varepsilon \neq 0$, Figure~\ref{fig:bifdiageps} illustrates some typical profiles, one per region defined by the GSPT analysis. The values chosen for $\alpha$ and $\beta$ are the same as in Figure~\ref{fig:bifdiag}, the solutions are qualitatively very close.

\begin{figure}[!ht]
    \centering
    \includegraphics[scale=1]{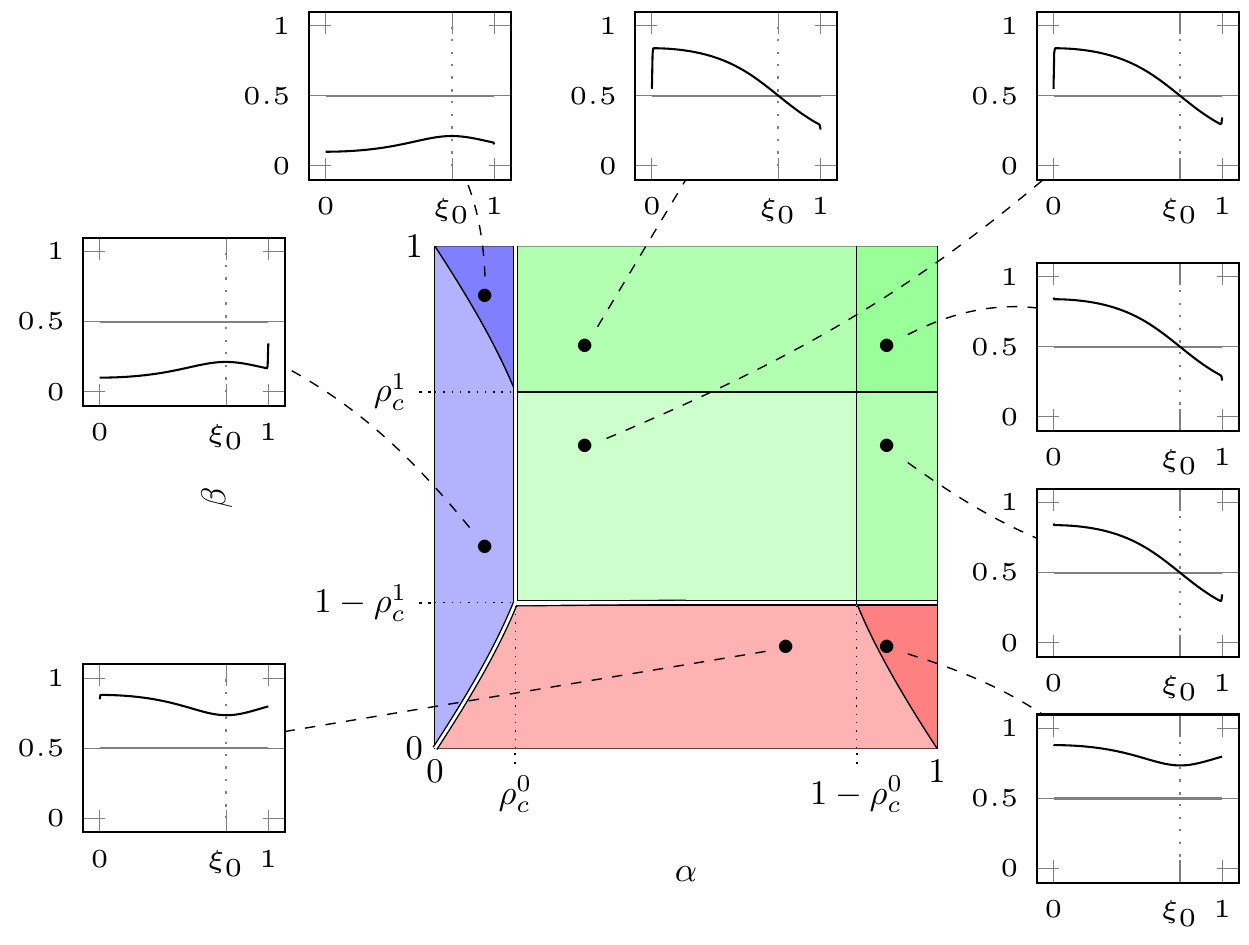}
    \caption{$(\alpha,\beta)$ bifurcation diagram for $\varepsilon=10^{-3}$. Here $k(\xi)=1+a\,\mathrm{cos}\left(\frac{2 \pi \xi}{b} \right)$ with $a=0.3$, $b=1.5$. The insets show the density $\rho$ as a function of $\xi$.}
    \label{fig:bifdiageps}
\end{figure}

Next investigate more realistic choices for $k$, which should mimic a corridor with a  bottleneck. We consider two regions of constant width that are connected by a narrower section in the middle. In particular, we consider a ``supergaussian'' profile for $k$:\[
k(\xi) = w_e - (w_e - w_m) e^{-\left|\frac{\xi-\xi_0}{d}\right|^6}\,,
\]
where $w_e$, $w_m$, $d$ and $\xi_0$ are positive parameters, corresponding to the width of the wider regions at the left and right, the width of the narrow middle section, the neck length and the neck position, respectively. We pick $w_e = 1$ and consider both $w_m = 0.9$, corresponding to wider neck and $w_m = 0.5$ which gives a more pronounced neck. The other parameters are taken as $w_e = 1$, $d = 0.2$, $\xi_0 = 0.6$, which gives a satisfactory, asymmetric width profile as shown in Figure~\ref{fig:supergaussian channel profile}.

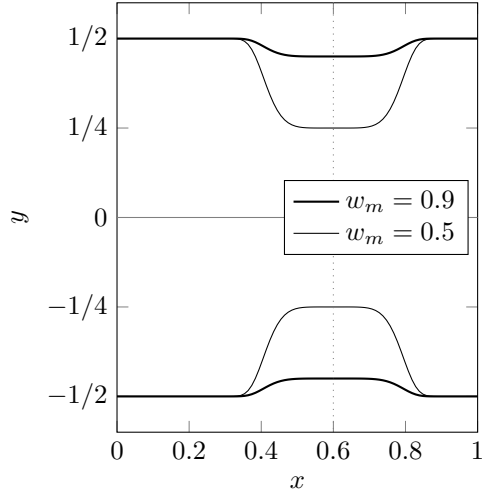
\begin{figure}[ht]
    \centering
    \pgfplotsset{samples=100,
        every axis legend/.append style={
            at={(0.975,0.5)},
            anchor=east,
        }
    }
    \begin{tikzpicture}
    \begin{axis}[xlabel=$x$,ylabel=$y$, xmin=0, xmax=1, ymin=-0.6, ymax=0.6, unit vector ratio=1,
        ytick={-0.5, -0.25, 0, 0.25, 0.5},
        yticklabels={$-1/2$, $-1/4$, $0$, $1/4$, $1/2$}]
    \addplot[domain=0:1,thick] {0.5*(1 - (1 - 0.9)*exp(-((x-0.6)/0.2)^6))};
    \addlegendentry{$w_m = 0.9$}
    \addplot[domain=0:1,thin] {0.5*(1 - (1 - 0.5)*exp(-((x-0.6)/0.2)^6))};
    \addlegendentry{$w_m = 0.5$}
    \addplot[domain=0:1,thick] {0.5*(-1 + (1 - 0.9)*exp(-((x-0.6)/0.2)^6))};
    \addplot[domain=0:1,thin] {0.5*(-1 + (1 - 0.5)*exp(-((x-0.6)/0.2)^6))};
    
    \addplot[domain=0:1,gray,thin] {0};
    \addplot[domain=-1:1,gray,thin,dotted] ({0.6},x);
    \end{axis}
    \end{tikzpicture}
    \caption{Representation of the 2D domain associated with a supergaussian $k$. The thick and bold lines correspond to $w_m = 0.5$ and $w_m = 0.9$, respectively.}
    \label{fig:supergaussian channel profile}
\end{figure}

Some characteristic profiles are shown in Figure~\ref{fig:bifdiageps supergauss}, along with the 8 regions defined by the GSPT analysis above. The selected values of the parameter pair $\alpha$ and $\beta$ are picked with exactly one pair value per region, and the same for both $w_m = 0.9$ and $w_m = 0.5$. The parameters $\alpha$ and $\beta$ are also chosen away from the $0$ and $1$, since those values lead to almost constant solutions for regions $\mathcal{G}_1$, $\mathcal{G}_2$, $\mathcal{G}_7$ and $\mathcal{G}_8$, which correspond to the blue and red shaded areas.
All chosen values are stated in Table~\ref{tab:supergauss params}.

\begin{table}
    \centering
    \begin{tabular}{cllllllll}
    \toprule
    &  \multicolumn{1}{c}{$\mathcal{G}_1$} &
       \multicolumn{1}{c}{$\mathcal{G}_2$}&
       \multicolumn{1}{c}{$\mathcal{G}_3$}&
       \multicolumn{1}{c}{$\mathcal{G}_4$}&
       \multicolumn{1}{c}{$\mathcal{G}_5$}&
       \multicolumn{1}{c}{$\mathcal{G}_6$}&
       \multicolumn{1}{c}{$\mathcal{G}_7$}&
       \multicolumn{1}{c}{$\mathcal{G}_8$}
    \\
    \midrule
    $\alpha$ & $0.119$  & $0.128$  & $0.5$    & $0.5$    & $0.881$  & $0.881$  & $0.5$    & $0.941$
    \\
    $\beta$  & $0.5$ & $0.941$ & $0.5$ & $0.881$ & $0.5$ & $0.881$ & $0.119$ & $0.128$ \\
    \bottomrule
    \end{tabular}
    \caption{Parameters for Figure~\ref{fig:bifdiageps supergauss}}
    \label{tab:supergauss params}
\end{table}

Generally speaking, we have three parameters ranges of interest, within which the stationary solutions share the same qualitative behaviour:
\begin{itemize}
    \item Small $\alpha$ (which corresponds to low inflow as in the blue regions $\mathcal{G}_1$ and $\mathcal{G}_2$), which leads to \emph{low} density stationary states with $\rho < \frac{1}{2}$ and a boundary layer on the right boundary.
    \item Small $\beta$ (which corresponds to low outflow as in the red regions $\mathcal{G}_7$ and $\mathcal{G}_8$), which leads to \emph{high} density stationary states with $\rho > \frac{1}{2}$ and a boundary layer on the left boundary.
    \item Large values of $\alpha$ and $\beta$ (corresponding to high inflow and outflow regimes as in the green regions $\mathcal{G}_3$ to $\mathcal{G}_6$), leading to density profiles going from high density on the left (before the bottleneck) to low density on the right (after).
    In this case, boundary layers a present on both boundaries.
\end{itemize}

Inside these three areas, solutions seem to depend only weakly on $\alpha$ and $\beta$, which affect the height of the boundary layers only.
It is only across the boundary between these areas (white lines) that pronounced qualitative changes occur.

\paragraph{Impact of the width of the bottleneck}

We now turn our attention to the influence of $w_m$ on the solutions. The first obvious difference in Figure~\ref{fig:bifdiageps supergauss} is the larger square in the center for small $w_m$, which corresponds to the region $\mathcal{G}_3$ in the singular analysis. This is explained by the very simple dependency of both $\rho_c^0$ and $\rho_c^1$ (which bound $\mathcal{G}_3$) on $k$, see \eqref{eq:sp_rho}.\\

\begin{figure}[!ht]
    \centering
    \includegraphics[scale=1]{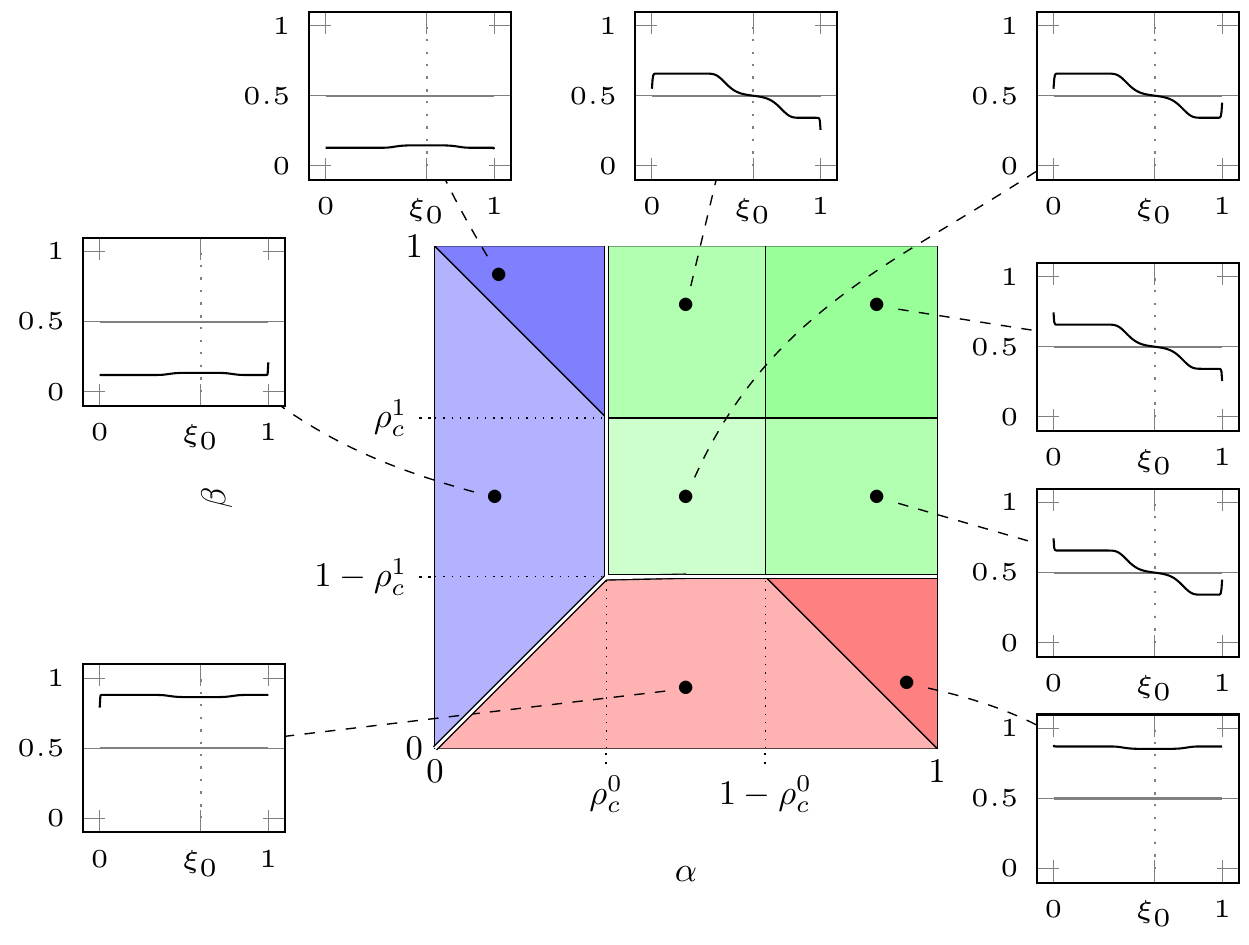}
    
    \vspace{1cm}
    \includegraphics[scale=1]{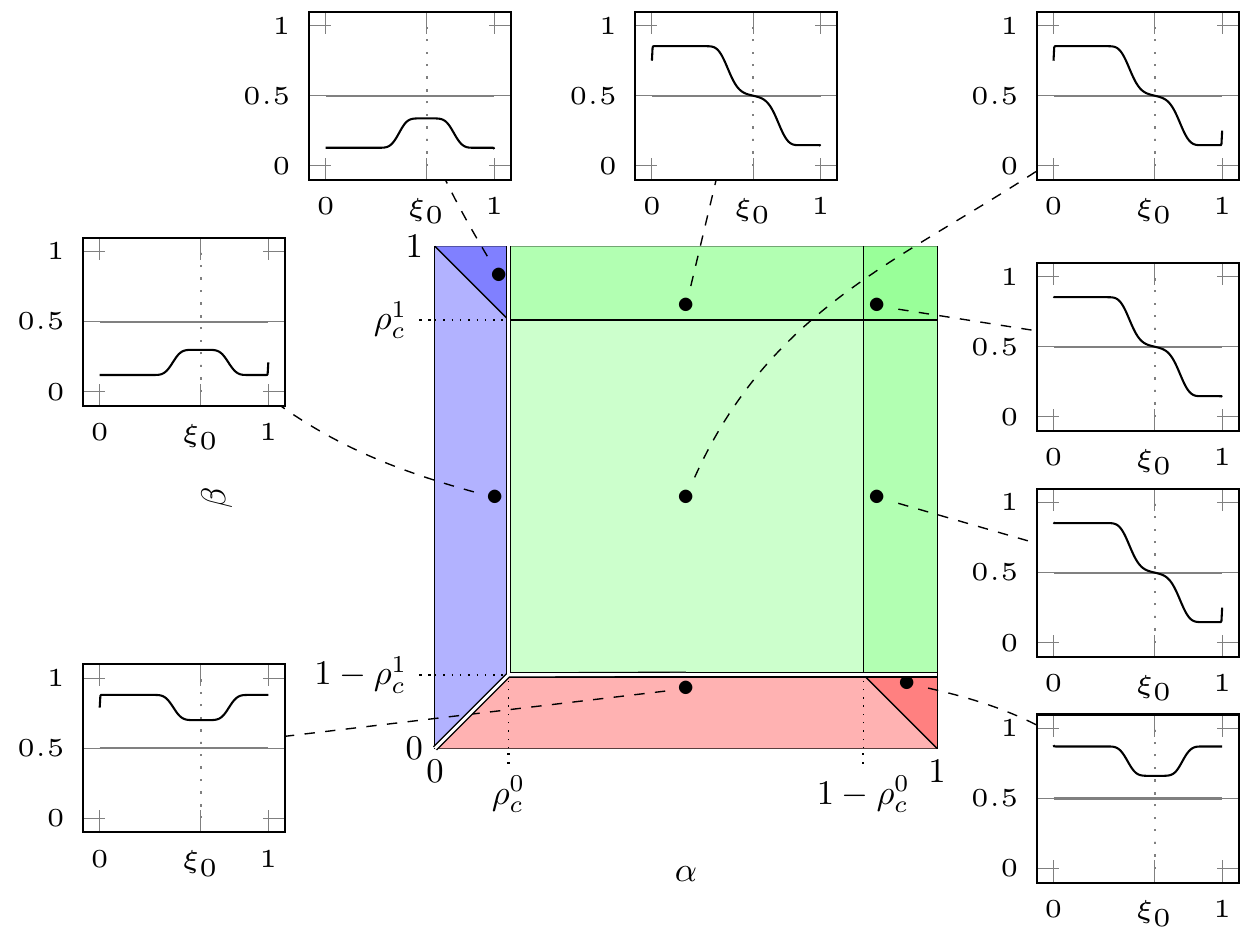}
    \caption{Phase diagrams in the $(\alpha, \beta)$ parameter space from the GSPT analysis for the supergaussian $k$ along with some typical \emph{non-singular} solutions for $w_m = 0.9$ (top) and $w_m = 0.5$ (bottom).}
    \label{fig:bifdiageps supergauss}
\end{figure}

In the regions of low (resp. high) density, in blue (resp. red) in Figure~\ref{fig:bifdiageps supergauss}, the density $\rho$ is roughly constant on large parts of the domain, with variations at the boundaries as well as at the front and back of the narrow section. Outside of it, $\rho$ takes similar values for both $w_m = 0.9$ and $w_m = 0.5$. Inside however, $\rho$ takes values much closer to $1/2$ for $w_m = 0.5$. Indeed, where $\rho$ is almost constant, the flux can be approximated as $\mathrm{J} = k \rho (1-\rho)$ ; since $\mathrm{J}$ is independent of $x$, lower values of $k$ correspond to $\rho$ closer to $1/2$. This also seems to indicate that in both the low and high density phases, the flux $\mathrm{J}$ for given $(\alpha, \beta)$ only depends weakly on $w_m$.

This numerical observation confirms the analytical results of Proposition \ref{prop:singsol} for the singular case ($\varepsilon=0$). We have in fact that in the regions of low (resp. high) density, studied in Case $1$, corresponding to regions $\mathcal{G}_1$ and $\mathcal{G}_2$ (resp. $3$, corresponding to $\mathcal{G}_7$ and $\mathcal{G}_8$), the density at the entrance (resp. exit) is given by $\alpha$ (resp. $\beta$) and hence is not affected by the features of the bottleneck. It follows that $\mathrm{J} = \alpha(1-\alpha)$ (resp. $\mathrm{J} = \beta(1-\beta)$).\\

In the green region (which could be argued to correspond to the so-called maximum flux phase for constant $k$), the situation is different. Although the profiles are qualitatively similar with a transition between a high density to a low density plateau, the densities for $w_m = 0.5$ (wider bottleneck, top) are much closer to $1/2$ for the values of $\alpha$ and $\beta$ which are considered. This relates to a higher flux for the wider bottleneck.

The computational results in these regions correspond to the analysis of Case $2$ (regions $\mathcal{G}_3$ to $\mathcal{G}_6$), a situation in which the density changes significantly (i.e. a boundary layer) in proximity of both the entrance and the exit, immediately preceded by a region where it is approximately constant. The density value in these areas is defined by $1-\rho_c^0$ and $1-\rho_c^1$, respectively. With the choice of parameters in Table~\ref{tab:supergauss params}, the approximation $k \simeq 1$ holds, at least in the first and last $10\%$ of the domain; we then get from its definition that $\rho_c^0$ (resp. $\rho_c^1$) is increasing (resp. decreasing) w.r.t $w_m$. In fact we have
\[
    1 - \rho_c^0 \simeq \rho_c^1 \simeq \frac{1}{2} \left( 1 + \sqrt{1 - w_m} \right) \quad \text{and} \quad \mathrm{J} \simeq \frac{1}{4}w_m\,,
\]
so that $\mathrm{J}$ grows linearly with the width of the neck and eventually reaches $1/4$ for $w_m = 1$, the maximum value for a straight channel.
This is in agreement with the numerical observations described above.\\

To summarize, the influence of the width of the neck in this case is two-fold. First, in terms of $(\alpha, \beta)$, the green region grows larger and eventually completely fills the parameter space as the neck-width goes to zero. Second, it is in this region that $w_m$ has a noticeable effect on the flux $\mathcal{J}$, which depends linearly on $w_m$, as one would expect intuitively.

\begin{remark}
The observations above are independent of the choice of $k$, provided that $k(0) = k(1) = 1$ and that the minimum of $k$ is non degenerate. In the singular case, one obtains an explicit expression for $\mathrm{J}$, which we write as $\mathrm{J}_{k}(\alpha, \beta)$ to emphasize the dependency on $k$, $\alpha$ and $\beta$:
\[
\mathrm{J}_k(\alpha, \beta) =
\begin{cases}
\alpha (1-\alpha) & \alpha \le \rho_c^0 \wedge \alpha \le \beta\,, \\
\beta (1-\beta) & \beta \le 1 - \rho_c^1 \wedge \beta \le \alpha\,, \\
\frac{1}{4} \min_\xi k(\xi) & \text{otherwise}\,.
\end{cases}
\]
In particular, we have that $\mathrm{J}_k(\alpha, \beta) = \min \left\{ \mathrm{J}_1(\alpha, \beta), \min_\xi k(\xi) \right\}$. This is illustrated in Figure~\ref{fig:flow geometry}. This means that as $\min_\xi k(\xi)$ decreases, the flow $\mathrm{J}_k(\alpha, \beta)$ will saturate, \emph{i.e.} reach its maximum, faster as $\alpha$ and $\beta$ increase. The maximum of $\mathrm{J}_k$ will also decrease linearly with $\min_\xi k(\xi)$. Numerical experiments with two narrow sections of varying width suggest that this applies also for functions $k$ with several (nondegenerate) critical points.
\end{remark}

\begin{figure}
    \begin{center}
        \includegraphics[width=0.4\textwidth]{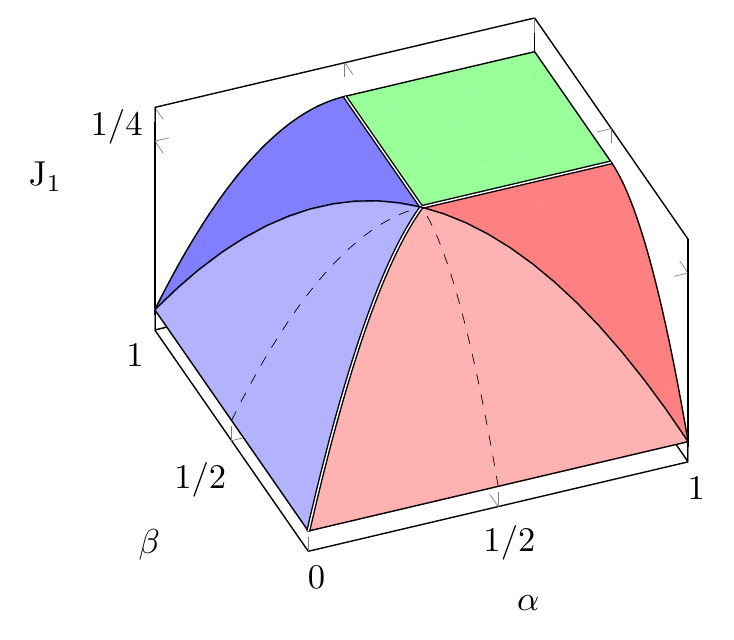}
        \hspace*{2cm}
        \includegraphics[width=0.4\textwidth]{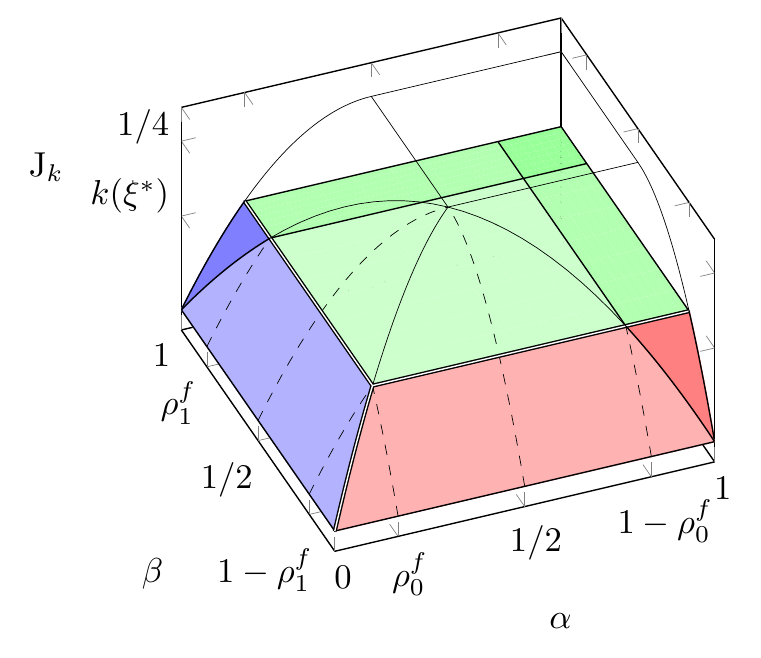}
    \end{center}
    \caption{Illustration of the flow $\mathrm{J}_1$ (solid colors on the left, wireframe on the right) and $\mathrm{J}_k$ (solid colors on the right) as a function of $\alpha$ and $\beta$. Recall that $\rho_c^{0,1} \rightarrow \frac{1}{2}$ as $k(\xi^*) \rightarrow 1$, so that from the green rectangles, only the darker one (top right) remains in the limit.}
    \label{fig:flow geometry}
\end{figure}

\section*{Conclusion}

%% What we did
% 1d reduction pedestrian model
% steady-states via GSPT
% bottlenecks: more complex behaviour, how geometry of domain influences bd
% influence of the neck
%% What we are going to do next
% Validation with data

In this work, we investigate the steady-states of a 1D area averaged model describing pedestrian dynamics for unidirectional flows in domains that have a bottleneck. In the proposed model, information about the geometry enters as a nonhomogeneous factor acting both on the diffusive and convective terms. We investigate the case in which this factor admits an isolated minimum, which corresponds to the bottleneck. The stationary profiles exhibit a multi-scale nature, which we analyse using GSPT. This allows us to thoroughly understand the influence of inflow and outflow rates ($\alpha$ and $\beta$, respectively) on the structure of the solutions and, in particular, on the formation of boundary layers. In this framework, the isolated minimum inside the bottleneck corresponds to a canard point where an unusual passage through a repelling branch of the critical manifold occurs. The more complex geometry therefore induces the emergence of two additional regions in the singular bifurcation diagram which have not been observed and investigated before. In general, orbits which include such passage exist for a wide area in the $(\alpha, \beta)$-parameter space, whose size decreases as the neck becomes wider. \\

In order to test the ability of our 1D reduction to capture the essential dynamics of the original two-dimensional model, we plan to suitably calibrate and validate our model as a next step. As observed in \cite{Iuorio_2022}, the quality of the proposed 1D area averaged approximation depends on the parameter regime considered; we will therefore investigate further averaging assumptions to overcome these issues in the next steps of our research.

\section*{Declaration of competing interests}
The authors declare no conflict of interest.

\section*{Acknowledgements}
AI acknowledges support from an FWF Hertha Firnberg Research Fellowship (T 1199-N).

\bibliography{ms}{}
\bibliographystyle{abbrv}

\clearpage

\end{document}